\documentclass[a4paper,12pt]{article}
\usepackage{amsmath,amssymb,amsthm}
\usepackage{hyperref}
\usepackage{verbatim}
\usepackage[british]{babel}
\usepackage[dvips]{graphics,color}

\usepackage{tikz}
\usepackage[export]{adjustbox}
\usepackage{float}

\def\spr{\color{red}}
\def\spc{\color{blue}}
\def\nc{\normalcolor}
\numberwithin{equation}{section}

\newtheorem{theorem}{Theorem}[section]
\newtheorem{lemma}[theorem]{Lemma}

\newtheorem{claim}[theorem]{Claim}
\newtheorem{criterion}[theorem]{Recurrence criterion}
\theoremstyle{definition}
\newtheorem{dfn}[theorem]{Definition}
\newtheorem{rem}[theorem]{Remark}

\newtheorem{example}[theorem]{Example}

\newcommand{\bx}{{\bf x}}

\newcommand{\R}{{\mathbb{R}}}
\renewcommand{\P}{\mathbb{P}}
\newcommand{\Z}{\mathbb{Z}}
\newcommand\N{\Z_+}

\newcommand{\E}{\mathbb{E}}
\newcommand{\gen}{\mathsf{L}}

\def\bn#1\en{\begin{align*}#1\end{align*}}
\def\bnn#1\enn{\begin{align}#1\end{align}}

\def\|{\, |\, }
\newcommand{\rmd}{\,\mathrm{d}}
\theoremstyle{definition}

\newenvironment{romenumerate}[1][-0pt]{
\addtolength{\leftmargini}{#1}\begin{enumerate}
 }{\end{enumerate}}
\newcommand\ga{\alpha}
\newcommand\gb{\beta}
\newcommand\gd{\delta}

\newcommand\gG{\Gamma}
\newcommand\kk{\kappa}
\newcommand\gl{\lambda}

\newcommand\go{\omega}
\newcommand\gO{\Omega}

\newcommand\bigpar[1]{\bigl(#1\bigr)}
\newcommand\Bigpar[1]{\Bigl(#1\Bigr)}

\newcommand\bigset[1]{\ensuremath{\bigl\{#1\bigr\}}}
\newcommand\set[1]{\ensuremath{\{#1\}}}
\newcommand\innprod[1]{\langle#1\rangle}

\newcommand{\refT}[1]{Theorem~\ref{#1}}

\newcommand{\refL}[1]{Lemma~\ref{#1}}
\newcommand{\refLs}[1]{Lemmas~\ref{#1}}
\newcommand{\refR}[1]{Remark~\ref{#1}}

\newcommand{\refS}[1]{Section~\ref{#1}}

\newcommand{\refSS}[1]{Section~\ref{#1}}

\newcommand{\refE}[1]{Example~\ref{#1}}
\newcommand{\refF}[1]{Figure~\ref{#1}}
\newcommand{\refrc}[1]{recurrence criterion~\ref{#1}}
\newcommand\qw{^{-1}}
\newcommand\qww{^{-2}}

\newcommand\bfe{{\bf e}}
\newcommand\bfv{{\bf v}}
\newcommand\bfu{{\bf u}}
\newcommand\bfo{{\mathbf 0}}
\newcommand\GO{\Gamma_0}
\newcommand\Reff{R_\infty}
\newcommand\Gabg{\Gamma_{\alpha, \beta, G}}
\newcommand\Gaog{\Gamma_{\alpha, 0, G}}
\newcommand\ZZn{\Z_+^n}
\newcommand\Zabg{Z_{\alpha, \beta, G}}
\newcommand\Zaog{Z_{\alpha, 0, G}}
\newcommand\hZabg{\hZ_{\alpha, \beta, G}}
\newcommand{\sumk}{\sum_{k=1}^\infty}

\newcommand{\sumL}{\sum_{L=1}^\infty}
\newcounter{case}
\newcommand\pfcasex[1]{\refstepcounter{case}\smallskip\noindent
  \emph{Case \arabic{case}: #1.}}
\newcommand\sumij{\sum_{i,j:\;i\sim j}}
\newcommand\sumin{\sum_{i=1}^n}
\newcommand\sumkn{\sum_{k=1}^n}

\newcommand\sfC{\mathsf{C}}
\newcommand\sfK{\mathsf{K}}
\newcommand\norm[1]{\lVert#1\rVert}
\newcommand\GGxx{\gG''}
\newcommand\hmu{\widehat\mu}
\newcommand\hZ{\widehat Z}
\newcommand\xxi{\bx}
\newcommand\grad{\nabla}
\newcommand\ax{a'}
\newcommand\indic{\boldsymbol{1}}
\newcommand\QZ{Q_\zeta}
\newcommand\DQZ{\Delta\QZ}
\newcounter{CC}
\newcommand{\CC}{\stepcounter{CC}\CCx} 
\newcommand{\CCx}{C_{\arabic{CC}}}     
\newcommand{\CCdef}[1]{\xdef#1{\CCx}}     
\newcommand{\CCname}[1]{\CC\CCdef{#1}}    
\newcommand\cE{\mathcal E}
\newcommand\abs[1]{|#1|}
\newcommand\tq{\widetilde q}
\newcommand\txi{\widetilde\xi}
\newcommand\tzeta{\widetilde\zeta}
\newcommand\tC{\widetilde C}
\newcommand\tGabg{\widetilde\Gamma_{\alpha, \beta, G}}
\newcommand\tZabg{\widetilde Z_{\alpha, \beta, G}}
\newcommand\exx[1]{(\emph{#1})}

\begin{document}
\title{Long term behaviour of a reversible system of  interacting random walks}
\author{
Svante Janson\footnote{Uppsala University. Email: svante.janson@math.uu.se},\quad
Vadim Shcherbakov\footnote{Royal Holloway, University of London. Email: vadim.shcherbakov@rhul.ac.uk}\quad 
and  Stanislav Volkov\footnote{Lund University.  Email: s.volkov@maths.lth.se}}
\maketitle

\begin{abstract}
This paper concerns the  long-term behaviour  of a system of interacting  random walks labeled by vertices of a finite graph.
The model  is reversible which allows to use the method of electric networks in the study. In addition, examples of alternative proofs not requiring  reversibility are provided. 
\end{abstract}

\noindent {{\bf Keywords:} Markov chain, random walk, transience, recurrence, Lyapunov function, martingale, 
renewal measure, return time.
}

\noindent {{\bf Subject classification:} 60K35, 60G50}

\section{Introduction}
\label{intro}

Let  $G$ be  a finite non-oriented  graph with   $n\ge1$ 
vertices labelled by $1,2,\dots, n$.
Somewhat abusing notation, we will use $G$ also for the set of the vertices of this graph. 
Let $A=(a_{ij})$ be the adjacency matrix of the graph, that is  $a_{ij}=a_{ji}=1$ or
 $a_{ij}=0$ according to whether vertices $i$  and $j$ are adjacent
 (connected by an edge)  or not.
If  vertices $i, j\in G$ are connected by an edge, i.e.\ $a_{ij}=1$, call them {\it neighbours} and write $i\sim j$. 
By definition, a vertex is not a
 neighbour of itself, i.e $a_{ii}=0$ for all $i=1,\dots,n$ (i.e.\ there are no self-loops).

Let $\Z_{+}$ be the set of all non-negative integers including zero. 
Consider a continuous-time Markov chain, CTMC for short, 
$\xi(t)=(\xi_1(t),\ldots, \xi_n(t))\in \Z_{+}^{n}$ evolving as follows. 
Given $\xi(t)=\xi=(\xi_1,\dots, \xi_n)\in \Z_{+}^{n}$,
 a component $\xi_i$ increases by~$1$ at the rate 
 \begin{equation}
   \label{rate}
e^{\alpha\xi_i+\beta(A\xi)_i}=e^{\alpha\xi_i+\beta\sum_{j: j\sim i}\xi_j},
 \end{equation}
 where $\alpha, \beta\in \R$ are two given constants.
A positive component~$\xi_i$ decreases by~$1$ at constant rate~$1$. 

In other words,  denoting the rates of the CTMC $\xi(t)$ by $q_{\xi,\eta}$, for $\xi,\eta\in\ZZn$, we have 
\begin{equation}
  \label{q}
q_{\xi,\eta}=
\begin{cases}
  e^{\alpha\xi_i+\beta(A\xi)_i}=e^{\alpha\xi_i+\beta\sum_{j: j\sim i}\xi_j}, 
& \eta=\xi+\bfe_i, \\
1, & \eta=\xi-\bfe_i, \\
0, & \norm{\eta-\xi}\neq1,
\end{cases}
\end{equation}
where ${\bf e}_i\in \Z_{+}^n$ is
the $i$-th unit vector, and
$\lVert\cdot\rVert$ denotes the usual Euclidean norm.


It is easy to see that if~$\beta=0$,  then CTMC $\xi(t)$  is  a collection
of $n$   independent reflected 
continuous-time
random walks  on $\Z_{+}$
(symmetric if also $\alpha=0$).
In general, 
the Markov chain can be regarded as an inhomogeneous random walk on infinite graph $\Z_{+}^G$. 
Alternatively, it can  be   interpreted   as a system  of $n$ random 
walks on $\Z_{+}$  labelled by the vertices of graph $G$ and 
evolving   subject to a nearest neighbour interaction. 

The purpose of the present paper is to study how the long term behaviour of
CTMC $\xi(t)$ depends on the parameters $\alpha$ and $\beta$ together with
 properties of the graph $G$. 
In our main result (\refT{Tmain}), we give a complete classification saying 
whether 
the Markov chain is recurrent or transient, and
in the recurrent case  whether it is positive recurrent or
null recurrent.
We find phase transitions, with different behaviour in various regions 
depending of parameters $\alpha$, $\beta$ and properties of  graph $G$.
Furthermore,
we give results  (\refT{Texpl})
on whether the Markov chain is explosive or not.
(This is relevant for the transient case only, since
a recurrent CTMC always is non-explosive.)
These results are  less complete and leave one case open.

It is obvious that CTMC $\xi(t)$ is irreducible; hence the initial
distribution is irrelevant for our results. (We may if we like  assume that we
start at $\bfo=(0,\dots,0)\in\Z_+^n$.)

CTMC $\xi(t)$ was introduced in~\cite{Volk3}, where
its long term behaviour  was studied  in several cases.
In particular, conditions for positive or null recurrence and transience  
were obtained in some special cases; these results are extended in the
present paper.
In addition, 
the typical asymptotic behaviour of the Markov chain was studied in some
transient cases. 

One  example of our results is the case $\alpha<0$ and
$\beta>0$, which  is of a particular 
interest because of the following phenomenon observed in \cite{Volk3} in
some special cases.
If $\alpha<0$ and $\beta=0$, then, as said above, CTMC $\xi(t)$  is formed
by a collection of independent 
positive recurrent reflected random walks on $\Z_{+}$, and is thus positive recurrent.
If both  $\alpha<0$ and $\beta<0$, 
then the Markov chain is still positive recurrent (as shown below). The interaction in this case is, in a sense, competitive, as 
neighbours obstruct the growth of each other. 
Now keep $\ga<0$ fixed but let $\gb>0$.
If $\beta$ is positive, but  not large,  then one could intuitively expect that the Markov chain is still positive recurrent 
(``stable''), as the interaction (cooperative in this case) is not strong enough.
On the other hand, if $\beta>0$ is sufficiently large,  then  the intuition suggests that 
the Markov chain becomes transient (``unstable").
It turns out that this is correct and that the  phase transition in the
model behaviour occurs at  the  critical value 
$\beta=\frac{|\alpha|}{\lambda_1(G)}$, where $\lambda_1(G)$ is the largest eigenvalue of graph $G$.
Namely, if $\beta<\frac{|\alpha|}{\lambda_1(G)}$ then the Markov chain is positive recurrent, and if 
$\beta\geq \frac{|\alpha|}{\lambda_1(G)}$ then the Markov chain is transient.
Moreover, it turns out that exactly at the critical regime, 
i.e.,  $\beta=\frac{|\alpha|}{\lambda_1(G)}$, the Markov chain 
is non-explosive transient.
 We conjecture that if $\beta>\frac{|\alpha|}{\lambda_1(G)}$,
then it is explosive transient. This remains as an open problem in the
general case (see Remark~\ref{Rexpl} below).
Another important contribution of this paper to the previous study of the Markov chain 
is a recurrence/transience classification in the case~$\alpha=0$ and~$\beta<0$.
This case was  discussed in~\cite{Volk3} only for the simplest graph with two vertices.
We show that in general there are only two possible long term behaviours of the Markov chain if~$\alpha=0$ and~$\beta<0$.
Namely, CTMC $\xi(t)$ is either non-explosive transient or null recurrent, and 
this  depends only on the independence number of the graph~$G$. 

We also consider some variations of the Markov chain defined above.
First, we include in our results
the Markov chain above with dynamics 
 obtained by setting $\beta=-\infty$ (with convention $0\cdot \infty=0$).
In other words, 
a component cannot jump up (only down, when possible),
if at least one of its neighbours is non-zero;
this can thus be interpreted as \emph{hard-core} interaction. 
 See \refSS{SShardcore} for more details on this hard-core case.

In \refS{S:DTMC} we consider the discrete time Markov chain (DTMC)  
$\zeta(t)\in \Z_{+}^n$ that corresponds to 
CTMC $\xi(t)$, i.e.\ the corresponding embedded DTMC. 
We show that our main results also apply to this DTMC.

Finally, in \refS{Smodified}, we study the CTMC with the rates given by
\begin{equation}
  \label{tq}
\tq_{\xi,\eta}=
\begin{cases}
  e^{\alpha\xi_i}, 
& \eta=\xi+\bfe_i, \\
e^{-\beta\sum_{j: j\sim i}\xi_j},  & \eta=\xi-\bfe_i, \\
0, & \norm{\eta-\xi}\neq1.
\end{cases}
\end{equation}
We show that similar results holds for this chain, although there is a minor
difference.

We use essentially the method of electric networks in our proofs; this is
possible since the
 CTMC $\xi(t)$ is reversible (see \refSS{SSrev}). 
The use of reversibility was rather limited  in~\cite{Volk3},
where  the Lyapunov function method and  direct probabilistic arguments were the  main research techniques.
In addition, we provide  examples of alternative proofs of some of our
results based on the Lyapunov function method  
and renewal theory for random walks. The advantage of these  alternative
methods
is that they do not require reversibility and can be applied in more general situations.
Therefore, the alternative proofs are of interest on their own right.

\begin{rem}\label{Rsimple}
  In the case $\ga=\gb=0$, all rates \eqref{rate} equal 1, and the Markov
  chain is  a continuous-time simple random walk on $\ZZn$.
It is known that a simple random walk on the octant $\Z_+^n$ is 
null recurrent for $n\le2$ and transient for $n\ge3$;
this is a variant of the corresponding well-known result for simple random
walk on $\Z^n$, and can rather easily be shown using electric network theory,
see \refE{Esimple} below.
\end{rem}

\begin{rem}\label{Rdisconnected}
  We allow the graph $G$ to be disconnected.
However, there is no
interaction between different
components of $G$, and the CTMC $\xi(t)$ consists of independent Markov
chains defined by the connected components of $G$. Hence, the case of main
interest is when $G$ is connected.
\end{rem}

\begin{rem}\label{Rempty}
The case when $G$ has no edges is somewhat exceptional but also rather trivial, 
since then the value of $\beta$ is irrelevant, and $\xi(t)$ consists of $n$
independent continuous-time  random walks on $\Z_+$; 
in fact, $\xi(t)$ then is as in
the case $\gb=0$ for any other $G$ with $n$ vertices.
In particular, if $G$ has no edges, we may assume $\gb=0$.
\end{rem}

\begin{rem}
  CTMC $\xi(t)$ is a model of interacting spins and, as such,  is related to models of statistical physics.
The stationary distribution of a finite Markov chain with 
 bounded components  and the same transition rates 
is of interest in statistical physics. In particular, 
if  components take only values $0$ and $1$, then the stationary distribution 
of the corresponding Markov chain is equivalent to a special case of  the
famous Ising model.
One of the main problems in statistical physics is to determine 
whether such probability distribution is subject to phase transition 
as the underlying graph indefinitely expands.
In the present paper, we keep the finite graph $G$ fixed, but allow arbitrarily
large components $\xi_i$. We then study phase transitions of this model, in
the sense discussed above.
\end{rem}


\section{The main results}\label{Smain}

In order to state our results, we need two definitions from graph theory.
We also let $e(G)$ denote the number of edges in $G$. 

\begin{dfn}
\label{D0}
The \emph{eigenvalues} of a finite graph $G$ are the eigenvalues of its
adjacency matrix $A$. These are real, since $A$ is symmetric, and we denote
them by $\lambda_1(G)\geq \lambda_2(G)\geq\dots\geq  \lambda_n(G)$, so that
$\lambda_1:=\lambda_1(G)$ is the largest eigenvalue. 
\end{dfn}
Note that $\gl_1(G)>0$
except in the rather trivial case $e(G)=0$
(see \refR{Rempty}).

\begin{dfn}
\label{D1}
\begin{romenumerate}
\item 
An \emph{independent set} of vertices in a graph $G$ is a set of the
vertices such that  
no two vertices in the set are adjacent.
(I.e., no pair of vertices in the set are joined by an edge of $G$.)
\item 
The \emph{independence number} $\kappa= \kappa(G)$ of a graph $G$
is  the cardinality of the largest independent  set of vertices.
\end{romenumerate}
\end{dfn}
For example, if $G$ is a cyclic graph $\sfC_n$
with $n$ vertices, then $\kappa=\lfloor n/2 \rfloor$.

The main results of the paper are collected in the following theorem,
which generalises results concerning positive recurrence of the Markov chain
obtained in~\cite{Volk3}. 
\begin{theorem}\label{Tmain}
Let $-\infty<\ga<\infty$ and $-\infty\le\gb<\infty$, and 
consider the CTMC\/ $\xi(t)$.
  \begin{romenumerate}
  \item \label{Tmain+}
If\/ $\ga<0$ and $\ga+\gb\gl_1(G)<0$, then  $\xi(t)$ is positive
recurrent.

  \item \label{Tmain0}
$\xi(t)$ is null recurrent in the following cases:
    \begin{enumerate}
    \item \label{Tmain0a}
$\ga=0$, $\gb<0$ and $\kk(G)\le2$,
\item \label{Tmain0b}
$\ga=\gb=0$ and $n\le2$,
\item \label{Tmain0c}
$\ga=0$, $\gb>0$, $e(G)=0$ and $n\le2$.
    \end{enumerate}

\item \label{Tmaint}
In all other cases,    $\xi(t)$ is transient. This means the cases 
\begin{enumerate}
\item\label{Tmainta}  
$\ga>0$,
\item $\ga=0$, $\gb>0$ and $e(G)>0$,
\item $\ga=0$, $\gb>0$, $e(G)=0$ and $n\ge3$,
\item $\ga=\gb=0$ and $n\ge3$,
\item \label{Tmaintkk}
$\ga=0$, $\gb<0$ and $\kk(G)\ge3$,
\item \label{Tmaint-} 
$\ga<0$ and $\ga+\gb\gl_1(G)\ge0$.
\end{enumerate}
  \end{romenumerate}
\end{theorem}

\refT{Tmain} is summarized in the diagram in \refF{fig:1}

\begin{figure}[h]
  \centering
\setlength{\unitlength}{1cm}
\begin{picture}(14,16)
\put(8, 16){\line( 0, -1){2}}
\put(8, 16){\line( 3, -1){6}}
\put(8, 16){\line(-3, -1){6}}
\put( 4.3, 14.5){$\alpha<0$}
\put( 8.1, 14.5){$\alpha=0$}
\put(12.7, 14.5){$\alpha>0$}
\put(2, 14){\line(-2, -3){2}}
\put(2, 14){\line(+2, -3){2}}
\put(0.5,   11.5){$\alpha+\beta\lambda_1<0$}
\put( 3.7, 11.5){$\alpha+\beta\lambda_1\ge 0$}
\put(8, 14){\line(-2, -3){2}}
\put(8, 14){\line(+5, -2){7}}
\put(8, 14){\line(+1, -1){3}}
\put(6.5,   11.5){$\beta< 0$}
\put(14.5,   11.5){$\beta> 0$}
\put(10.6,  11.5){$\beta=0$}
\put(  6,  11){\line(-1, -1){1.5}}
\put(  6,  11){\line(+1, -1){1.5}}
\put(  5, 9.8){$\kappa\le 2$}
\put(7.3, 9.8){$\kappa\ge 3$}
\put(  11,  11){\line(-1, -1){1.5}}
\put(  11,  11){\line(+1, -1){1.5}}
\put(10.1, 9.8){$n\le 2$}
\put(12.3, 9.8){$n\ge 3$}
\put(  14, 13.5){\oval(2,1)}
\put(13.5, 13.5){Trans.}
\put(   0, 10.5){\oval(2,1)}
\put(-0.7, 10.5){Pos.Rec.}
\put(4,   10.5){\oval(2,1)}
\put(3.3, 10.5){Trans.}
\put(14.8, 10.7){\oval(2,1)}
\put(14.3, 10.7){Trans.} 
\put(4.5, 9){\oval(2,1)}\put(3.7, 9){Null Rec.}
\put(7.5, 9){\oval(2,1)}\put(6.8, 9){Trans.}
\put( 9.7, 9){\oval(2,1)}\put( 8.9, 9){Null Rec.}
\put(12.7, 9){\oval(2,1)}\put(11.9, 9){Trans.}
\end{picture}
\vskip-8cm
\caption{The different phases for the CMTC $\xi(t)$. (Ignoring a trivial exception if $e(G)=0$, $\ga=0$ and $\gb>0$,  see~\refR{Rempty}.)}
\label{fig:1}
\end{figure}
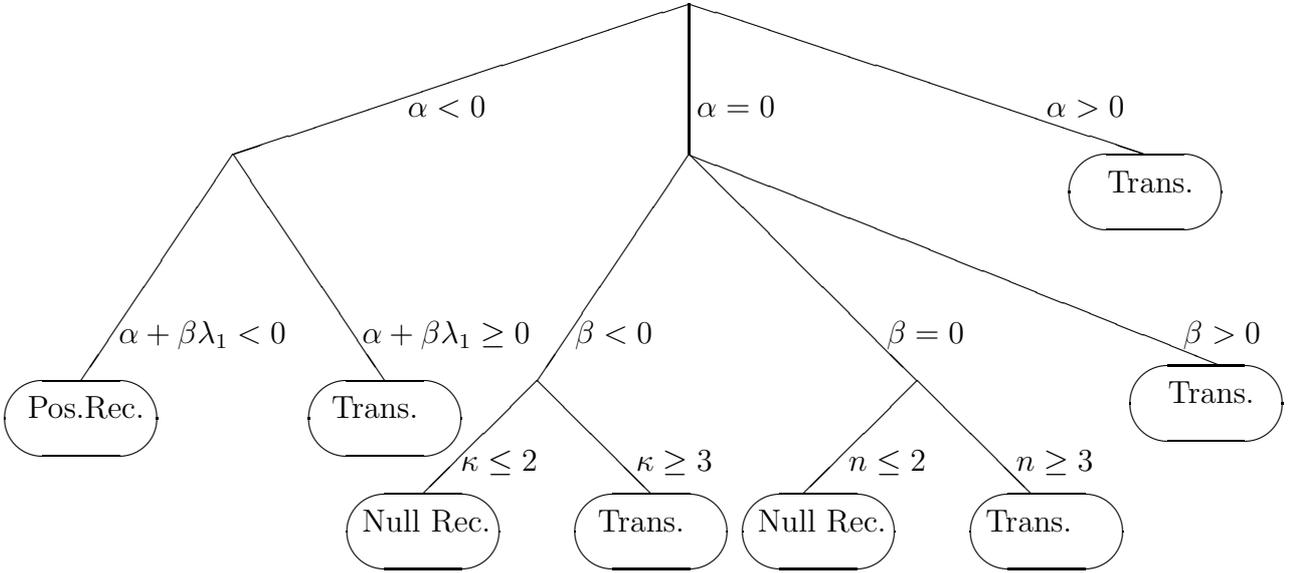

\begin{rem}
  \label{Rmono}
\refT{Tmain} shows that the behaviour of the Markov chain has the
following monotonicity property: 
if the Markov chain is transient for some given parameters 
$(\alpha_0, \beta_0)$, then it is also transient for all parameters
$(\alpha, \beta)$ 
such that $\alpha\geq \alpha_0$ and $\beta\geq \beta_0$. 
This can also easily be seen directly using electric networks as in
\refSS{SSelectric}, see the proof of \refL{LR1b}.
\end{rem}

\begin{rem}
\label{RCor1}
 There is a vast literature devoted to a graph eigenvalues. In particular,
 there are well known bounds for  the largest eigenvalue $\gl_1$.
We give two simple examples where 
the largest eigenvalue $\gl_1$ easily can be computed
explicitly,
which allows us to rewrite the
conditions of Theorem~\ref{Tmain} in the case $\ga<0$ in more explicit form.
These examples  basically rephrase  results previously obtained in
\cite[Theorems 4 and 6]{Volk3}. 
\end{rem}

\begin{example}
 Assume that $G$ is a regular graph, i.e., a graph with constant vertex
 degrees $\nu$, say.
Then $\gl_1=\nu$.
Hence, 
the Markov chain is positive recurrent if and
 only if $\ga<0$ and $\alpha+\beta\nu<0$. 
If $\alpha<0$ and 
$\alpha+\beta\nu\geq 0$, then
the Markov chain is transient.
\end{example}

\begin{example}\label{EK1m}
 Assume that the graph  $G$ is  a star $\sfK_{1,m}$ with $m=n-1$ non-central
 vertices, where $m\geq 1$. 
A direct computation gives that $\lambda_1=\sqrt{m}$.
Hence, 
the Markov chain is positive recurrent if and
 only if $\ga<0$ and $\alpha+\beta\sqrt{m}<0$. 
If $\alpha<0$ and 
$\alpha+\beta\sqrt{m}\geq 0$, then
the Markov chain is transient.
\end{example}

We consider also two examples with $\ga=0$ and $\gb<0$, when the
independence number $\kk(G)$ is crucial.

\begin{example}\label{EK1m0}
 Let, as in \refE{EK1m},  $G$ be  a star $\sfK_{1,m}$, where $m\geq 1$. 
Then $\kk(G)=m=n-1$.
Assume that $\ga=0$ and $\gb<0$.
Then,
the Markov chain is null recurrent if
$n\le3$,
and transient if $n\ge4$.
\end{example}

\begin{example}\label{EC0}
 Let  $G$ be  a cycle $\sfC_{n}$, where $n\geq 3$. 
Then $\kk(G)=\lfloor n/2\rfloor$.
Assume that $\ga=0$ and $\gb<0$.
Then,
the Markov chain is null recurrent if
$n\le5$,
and transient if $n\ge6$.
\end{example}

\section{Preliminaries}\label{S:prel}

\subsection{Reversibility of the Markov chain} \label{SSrev}
Define the following function 
\begin{equation}
\label{W}
W(\xi):=
\frac{\alpha}{2}\sum\limits_{i=1}^n\xi_i(\xi_i-1)+\beta\sumij\xi_i\xi_j
=
 \frac{1}{2}\innprod{(\alpha E+\beta A)\xi, \xi}-\frac{\alpha}{2}S(\xi),
\quad \xi\in\Z_{+}^n,
\end{equation}
where 
the second sum is interpreted as the sum over unordered pairs $\set{i,j}$
(i.e, a sum over the edges in $G$),
$\innprod{\cdot, \cdot}$ is the Euclidean scalar product, $E$ is the unit $n\times n$ matrix, $A$ is the
 adjacency matrix of the graph $G$  and 
\begin{equation}
\label{S}
S(\xi):=\sum\limits_{i=1}^n\xi_i.
\end{equation}
A direct computation gives  the  detailed balance equation
\begin{equation}
\label{balance}
e^{\alpha\xi_i+\beta(A\xi)_i}e^{W(\xi)}=e^{W\left(\xi+{\bf e}_i\right)},
\end{equation}
for $i=1,\dots, n$ and $\xi\in \Z_{+}^{n}$.
Note that, recalling  \eqref{q},
\eqref{balance} is equivalent to the standard form of the balance
equation
\begin{equation}
  \label{balanceq}
q_{\xi,\eta} e^{W(\xi)} = q_{\eta,\xi} e^{W(\eta)}, 
\qquad \xi,\eta\in\ZZn.
\end{equation}
Hence, (\ref{balance}) 
 means that 
the Markov chain is reversible with invariant measure $\mu(\xi):=e^{W(\xi)}$,
$\xi\in\Z_{+}^n$.

The explicit formula for the invariant measure $\mu$ enables us to easily
see when $\mu$ is summable,
and thus can be normalised to
an invariant distribution (i.e., a probability measure); we return to this in
\refL{LI}.

\begin{rem}
  Recall that a recurrent CTMC has an invariant measure that is
unique up to a multiplicative constant, while a transient CTMC in general
may have
several linearly independent invariant measures (or none). 
We do not investigate whether the invariant measure
$\mu$ is unique (up to constant factors) for our Markov
chain also in transient cases.
\end{rem}

\subsection{Electric network corresponding to the Markov chain}
\label{SSelectric}

Let us define the electric network on graph $\Z_{+}^n$ corresponding to the Markov chain of interest.
According to the general method (e.g., see~\cite{DS} or~\cite{Kelly}) the construction goes as follows. 
First, suppose that $\beta>-\infty$. Given $\xi=(\xi_1,\dots, \xi_n)\in \Z_{+}^n$ replace each edge 
$$
\{\xi-{\bf e}_i, \xi  \}=\left\{(\xi_1,\xi_2,\dots, \xi_{i}-1,\dots, \xi_n), \,
(\xi_1, \xi_2,\dots, \xi_i,\dots, \xi_n)\right\}, \quad i=1,\dots,n,
$$
(assuming $\xi_i\ge 1$) by a resistor with conductance (resistance${}^{-1}$)
equal to
\begin{equation}
\label{eqres}
C_{\xi-{\bf e}_i, \xi}:=e^{W(\xi)}.
\end{equation}
 Note that  $C_{\xi-{\bf e}_i, \xi}$  does not depend on $i$ in our case. 
Also, $C_{\bfo,{\bfe}_i}=e^{W({\bf e}_i)}=1$, i.e., the edges connecting the
origin  $\bfo$
with $\bfe_i$ have conductance 1, and thus resistance
 $1$ (Ohm, say).

We  denote the network consisting of $\ZZn$ with the conductances
\eqref{eqres} by $\Gabg$.
Otherwise, we will for convenience
sometimes denote an electric network by the same symbol as the underlying graph
when it is clear from the context what the conductances are.

Let  $N(\Gamma)$ be an electric network on an infinite graph $\Gamma$.
The  effective resistance $\Reff(\Gamma)=\Reff(N(\Gamma))$ 
of the network is defined, loosely speaking,  as the resistance between some
fixed point of $\Gamma$, which in our case we choose as $\bfo$,  and infinity 
(see e.g.~\cite{DS}, \cite{Kelly} or \cite{Lyons-Peres} for more details).
Recall  that  a reversible Markov chain is transient if and only if the effective resistance of the corresponding 
electric network is finite. 
Equivalently, a reversible Markov chain is recurrent if and only if the effective resistance of the corresponding 
electric network is infinite. 

A common approach to showing either recurrence or transience of a reversible
Markov chain is based on 
Rayleigh's monotonicity law.
In particular, if $N(\Gamma')$ is a subnetwork of $N(\Gamma)$, obtained by
deleting some edges, then $\Reff(\Gamma)\leq \Reff(\Gamma')$.
 Therefore, if 
$\Reff(\Gamma')<\infty$, then
$\Reff(\Gamma)<\infty$ as well, and thus the corresponding Markov chain
on $\Gamma$ is transient.
Similarly, if the network $N(\Gamma'')$ is obtained from $N(\Gamma)$ by
short-circuiting one or several sets of vertices, then  
$\Reff(\Gamma'')\leq \Reff(\Gamma)$.
Hence, if
$\Reff(\Gamma'')=\infty$, then
$\Reff(\Gamma)=\infty$ as well, and  the corresponding Markov chain
on $\Gamma$ is recurrent.

\begin{example}\label{Esimple}
We illustrate these methods, and  give a flavour of later proofs, 
by showing how they work for a simple random walk (SRW) on $\ZZn$, which as
said in \refR{Rsimple} is the special case $\ga=\gb=0$ of our model.
The corresponding electric network  has all resistances equal to $1$. 

First, we obtain a lower bound of $\Reff(\ZZn)$ by 
some short-circuiting.
(See \cite[page 76]{DS}, or  the Nash-Williams criterion and Remark 2.10
in \cite[pages 37--38]{Lyons-Peres}.)
Let, recalling \eqref{S},
\begin{equation}
  \label{VL}
V_L:=\set{x\in\Z_+^n:S(x)=L}, \qquad L=0,1,\dots,
\end{equation}
and let $\GGxx$ be the network obtained from $\ZZn$
by short-circuiting each set $V_L$
of vertices; we can regard each $V_L$ as a vertex in $\GGxx$.
Then  we have $\asymp L^{n-1}$ resistors {\em in parallel} connecting
$V_{L-1}$ and  $V_L$. 
As a result, their conductances (i.e.\  inverse of resistance) sum up; hence
the effective resistance $R_L$
between $V_{L-1}$ and $V_L$ is $\asymp \frac 1{L^{n-1}}$. Now $\GGxx$
consists of a sequence of resistors $R_L$ {\em in series}, so 
we must sum them; consequently the resistance of the modified network is 
\begin{equation}
\Reff(\GGxx)=\sumL R_L\asymp \sum_{L=1}^{\infty} \frac 1{L^{n-1}}.  
\end{equation}
If $n=1$ or $n=2$, this sum is infinite and thus
$\Reff(\ZZn)\ge\Reff(\GGxx)=\infty$; hence the SRW is recurrent.

On the other hand, if $n\ge3$, 
one can show that the random walk is transient.
See, for example, the description of the tree $NT_{2.5849}$ 
in \cite[Section 2.2.9]{DS},
or the construction of a flow with finite energy
in \cite[page 41]{Lyons-Peres} 
(there done for $\Z^n$, but works for $\ZZn$ too),
for a direct proof that $\Reff(\ZZn)<\infty$.
An alternative argument uses the well-known 
transience of SRW on $\Z^n$ ($n\ge3$) as follows.
Consider a unit current flow from $\bfo$ to infinity on $\Z^n$. By symmetry,
for every vertex $(x_1,\dots,x_n)\in\Z^n$, the potential is the same at all
points $(\pm x_1,\dots,\pm x_n)$. Hence we may short-circuit each such set
without changing the effective resistance $\Reff$. The short-circuited
network, $\gG'$ say, is thus also transient. However, $\gG'$ can be
regarded as a network on $\ZZn$ where each edge has a conductance between 2
and $2^n$ (depending only on the  number of non-zero coordinates). Hence, 
by Rayleigh's monotonicity law, $\Reff(\ZZn)\le 2^n\Reff(\gG')<\infty$,
and thus the SRW is transient.
\end{example}

\subsection{The hard-core interaction}\label{SShardcore}

Let us discuss in more detail  the model with hard-core interaction, i.e.\
$\beta=-\infty$.
Then a component $\xi_i$ can increase only when $\xi_j=0$ for every
$j\sim i$, and it follows that
the set
\begin{equation}
  \label{GO}
  \GO:=\{\xi\in \Z_{+}^{n}: \xi_i \xi_j=0 \text{ when } i\sim j\}
\end{equation}
is absorbing, i.e., if the Markov chain $\xi(t)$ reaches $\gO$, then it will
stay there forever. 
In particular, if the chain starts at $\bfo$, then it will stay in $\GO$.
Moreover, 
it is easy to see that 
given any initial state, the process will a.s.\ 
reach $\GO$ at some time
(and then thus stay in $\GO$).
Hence, any state $\xi\in\Z_{+}^n\setminus\GO$ 
(i.e., with at least two neighbouring non-zero components) 
is  a non-essential state, and the long-term behaviour of $\xi(t)$ depends
only on its behaviour on $\GO$. 

Therefore, in the hard-core case we consider the Markov chain 
with the state space $\GO$.
This chain on $\GO$ is easily seen to be irreducible.

Note that
$\GO$ is the set of  configurations such that $\innprod{A\xi,\xi}=0$, where
$A$ is 
the adjacency matrix of graph $G$.
Equivalently, a configuration $\xi$ belongs to $\GO$ if and only if
the set $\set{i:\xi_i>0}$ is an independent set of vertices in $G$
(see Definition~\ref{D1}).

\begin{rem}
  In the special case $\alpha=0$, the Markov chain with the
hard-core interaction 
$\beta=-\infty$
can be regarded as a simple symmetric random walk 
on the subgraph $\GO\subseteq\Z_+^n$. 
In this special case, \eqref{W} yields $W(\xi)=0$ for every $\xi\in\GO$, 
so by \eqref{eqres}, the conductance of every edge in $\GO$ is 1.
We may also regard this network as a network on $\Z_+^n$
with the conductance for edge $\{\xi-{\bf e}_k, \xi  \}$  defined by
\begin{equation}
\label{eqres1}
C_{\xi-{\bf e}_k, \xi}=
\begin{cases}
1,& \text{  if  }   \sumij\xi_i\xi_j=0,\\
0, & \text{  if  }   \sumij\xi_i\xi_j\neq 0,
\end{cases}
\end{equation}
where the second case  simply means that the edge is not wired.
\end{rem}


\section{Proof of \refT{Tmain}}\label{Spf}

In this section we prove \refT{Tmain} by proving a long series of lemmas
treating different cases.
Note that we include the hard-core case $\gb=-\infty$. (For emphasis we say
this explicitly each time it may occur.)
Recall that $\Gabg$ denotes $\ZZn$ regarded as an electrical network with
conductances \eqref{eqres} corresponding to the CTMC $\xi(t)$.

As a first application of the method of electric networks we treat the case
$\ga>0$.

\begin{lemma}\label{LT1}
 If\/ $\ga>0$ and $-\infty\le \gb <\infty$, then the CTMC $\xi(t)$ is transient.
\end{lemma}

\begin{proof}
Consider the subnetwork of $\Gabg$ consisting of the axis
$\Gamma'=\Z_+\bfe_1=\{\bx\in \Z_{+}^n:  x_i=0,\, i\neq 1\}$. 
For $k\geq 1$, the conductance on the edge connecting  $(k-1){\bf e}_1=(k-1, 0,\dots,0)$ and 
$k{\bf e}_1=(k, 0,\dots, 0)$ is by \eqref{eqres} and \eqref{W}
equal to 
\begin{equation}
  \label{LT1W}
e^{W(k\bfe_1)}=e^{\frac{\alpha}{2}k(k-1)};
\end{equation}
hence the resistance is $e^{-\frac{\alpha}{2}k(k-1)}$.
 Since the resistors in $\Gamma'$ 
are connected in series,  the effective resistance of this subnetwork is
\begin{equation}
\Reff(\Gamma')=\sum\limits_{k=1}^{\infty}e^{-\frac\alpha{2}{k(k-1)}}<\infty,
\end{equation}
as $\alpha>0$.
Therefore, the effective resistance of the original network $\Gabg$ is also
finite. Consequently, the Markov chain is transient.
\end{proof}

We give similar arguments for the other transient cases.
Recall that $A$ is a non-negative
symmetric matrix with eigenvalues $\gl_1,\dots,\gl_n$.
Thus there exists an orthonormal basis of eigenvectors  ${\bf v}_i$ with
$A{\bf v}_i=\lambda_i{\bf v}_i$, $i=1,\ldots, n$. By the Perron--Frobenius
theorem ${\bf v}_1$ can be chosen non-negative, i.e.\ ${\bf v}_1\in \R_{+}^n$. 
(If $G$ is connected, then $\bfv_1$ is unique and strictly positive.)

\begin{lemma}\label{LT2}
  If\/ $\ga<0$ and $\ga+\gb\gl_1\ge0$, 
 then the CTMC\/ $\xi(t)$ is transient.
\end{lemma}

\begin{proof} 
For each $t\geq 0$, define $x(t):=t{\bf v}_1$ and 
$y(t):=(\lfloor x_1(t) \rfloor, \ldots, \lfloor x_n(t) \rfloor)$.  By construction, $y(t)$ is piecewise constant. Let $y_0=0, y_1, y_2,\ldots$ be 
the sequence of different values of $y(t)$, where at each $t$ such that two or more coordinates of $y(t)$ jump simultaneously,
we insert intermediate  vectors, so that only one coordinate changes at a time,   and $\lVert y_{k+1}-y_k\rVert =1$ for 
all $k$. Then $S(y_k)$, the sum of coordinates of $y_k$, is equal to $k$, and thus $k/n\leq \lVert y_k\rVert \leq k$.
Furthermore, for each $k$ there is a $t_k$ such that 
$\lVert y_k-y(t_k)\rVert \le n$,
and thus 
\begin{equation}
  \label{yx}
\lVert y_k-t_k\bfv_1\rVert =\lVert y_k-x(t_k)\rVert =O(1).
\end{equation}
Express $y_k$ in the basis ${\bf v}_1,\ldots,\bfv_n$ as
$y_k=\sum_{i=1}^na_{k,i}{\bf v}_i$;
then \eqref{yx} implies $a_{k,i}=O(1)$ for $i\neq 1$. Thus
\begin{equation}
  \label{WYk0}
\innprod{(\alpha E+\beta A)y_k,
  y_k}=\sum\limits_{i=1}^n(\alpha+\beta\lambda_i)a^2_{k,i}=(\alpha+\beta\lambda_1)a_{k,
  1}^2+O(1)\geq O(1),
\end{equation}
since $\alpha+\beta\lambda_1\ge 0$ by assumption. Therefore,
by \eqref{W},
\begin{equation}
  \label{WYk}
  W(y_k)\geq -\frac{\alpha}{2}S(y_k)+O(1)=\frac{|\alpha|}{2}k+O(1).
\end{equation}
Consider the subnetwork $\Gamma'\subset \Gabg$ formed by the vertices
$\{y_k\}$. The resistance of the edge connecting $y_{k-1}$ and $y_k$
is equal to 
\begin{equation}
R_k=e^{-W(y_k)}\leq Ce^{- \frac{|\alpha|}{2}k},  
\end{equation}
so that the effective resistance of the subnetwork
$\Reff(\gG')=\sum_{k}R_k<\infty$.
Hence $\Reff(\Gabg)<\infty$ and the Markov chain is transient. 
\end{proof}

\begin{lemma}\label{LT3}
  If\/  $\ga=0$, $\gb>0$ and $e(G)>0$,
 then the CTMC\/ $\xi(t)$ is transient.
\end{lemma}

\begin{proof}
  We do exactly as in the proof of \refL{LT2} up to \eqref{WYk0}. Now
  $\ga=0$, so \eqref{WYk} is no longer good enough. Instead we note that
\eqref{yx} implies $a_{k,1}=t_k+O(1)$ and also
$k=S(y_k)=S(x(t_k))+O(1)=Ct_k+O(1)$,
where $C=S(\bfv_1)>0$. Thus, with $c=C^{-1}$, 
\begin{equation}\label{WY3}
a_{k,1}=
  t_k+O(1)
=ck+O(1).
\end{equation}
Furthermore, $\gl_1>0$ since $e(G)>0$, and thus \eqref{W},  \eqref{WYk0} and
\eqref{WY3} 
yield, recalling $\ga=0$, 
\begin{equation}
  \label{WYk2}
  W(y_k)
=\frac12 \gb\gl_1(ck+O(1))^2+O(1) \geq c_1k^2
\end{equation}
for some $c_1>0$ and all large $k$.

It follows again that the subnetwork $\Gamma':=\set{y_k}$ has finite effective
resistance, and thus the Markov chain is transient.
\end{proof}

Alternatively, several other choices of paths $\set{y_k}$ could have been
used in the proof of \refL{LT2}, for example $\set{(k,1,0,\dots,0):k\ge0}$.

\begin{lemma}\label{LT0}
  If\/  $\ga=0$, $\gb=0$ and $n\ge3$,
 then the CTMC\/ $\xi(t)$ is transient.
\end{lemma}

\begin{proof}
  As said in \refR{Rsimple} and \refE{Esimple}, 
in this case, the Markov chain is just simple
  random walk on $\ZZn$, which is transient for $n\ge3$.
\end{proof}

\begin{lemma}\label{LT0+}
  If\/  $\ga=0$, $\gb>0$, $e(G)=0$ and $n\ge3$,
 then the CTMC\/ $\xi(t)$ is transient.
\end{lemma}

\begin{proof}
  When $e(G)=0$, the parameter $\gb$ is irrelevant and may be changed to
  $0$.
The result thus follows from \refL{LT0}.
\end{proof}

\begin{lemma}\label{LT4}
  If\/  $\ga=0$, $\gb\ge-\infty$ and $\kk\ge3$,
 then the CTMC\/ $\xi(t)$ is transient.
\end{lemma}

\begin{proof}

Since $\kk\ge3$, 
there are three vertices of the graph $G$ not adjacent to each other;
 w.l.o.g.\ let them be $1$, $2$ and $3$.
Consider the subnetwork 
\begin{equation}
  \Gamma':=\Z_+^3\times\set{0}^{n-3}=\set{(\xi_1,\xi_2,\xi_3,0,\dots,0)}
\subset\GO
\subset\Gabg=\Z_+^n.
\end{equation}
By \eqref{W}, we have in this case $W(\xi)=0$ for every $\xi\in\gG'$, and thus
(\ref{eqres}) implies that in the corresponding electrical network 
all edges in $\gG'$ have conductance 1, and thus resistance 1.
Hence, the Markov chain corresponding to the network
$\gG'$ is simple random walk on $\gG'\cong\Z_+^3$.
By \refR{Rsimple} and \refE{Esimple}, 
a simple random walk on the octant $\Z_+^3$ is
transient, and thus  $\Reff(\gG')=\Reff(\Z_+^3)<\infty$.
Consequently,
$\Reff(\Gabg)\le\Reff(\gG')<\infty$, and thus the Markov chain is transient.
\end{proof}

We turn to proving recurrence in the remaining cases.

\begin{lemma}\label{LR1}
  If\/  $\ga<0$,  $\ga+\gb\gl_1<0$ and $\gb\ge0$,
 then the CTMC\/ $\xi(t)$ is recurrent.
\end{lemma}
  
\begin{proof}
  Let $\gd=-(\ga+\gb\gl_1)>0$.
The eigenvalues of the symmetric matrix $\ga E+\gb A$ are
$\ga+\gb\gl_i\le\ga+\gb\gl_1=-\gd$, $i=1,\dots,n$. Thus, by \eqref{W},
\begin{equation}
\label{WR1}
W(\xi)
=
 \frac{1}{2}\innprod{(\alpha E+\beta A)\xi, \xi}-\frac{\alpha}{2}S(\xi)
\le -\frac{\gd}2\innprod{\xi,\xi}+\frac{|\alpha|}{2}S(\xi)
.
\end{equation}

We now argue as in \refE{Esimple}.
Let again $V_L$ be defined by \eqref{VL},
and let $\GGxx$ be the network obtained from $\Gabg$
by short-circuiting each set $V_L$
of vertices.
For $\xi\in V_L$, we have by the Cauchy--Schwarz inequality
$L^2=S(\xi)^2\le n\innprod{\xi,\xi}$, and thus
by \eqref{WR1} and \eqref{eqres}, the conductance
\begin{equation}
\label{CR1}
C_{\xi-{\bf e}_i, \xi}
= e^{W(\xi)}
\le e^{-\frac{\gd}{2n}L^2+\frac{|\ga|}{2}L} \le Ce^{-cL^2}
\end{equation}
for some positive constants $c,C$.

For  $L\ge1$, there are $O(L^{n-1})$ vertices in $V_L$, and thus
$O(L^{n-1})$ edges between $V_{L-1}$ and $V_L$.
When short-circuiting each $V_L$, we can regard each $V_L$ as a single
vertex in $\GGxx$; the edges between $V_{L-1}$ and $V_L$ then become parallel,
and can be combined inte a single edge between $V_{L-1}$ and $V_L$.
The conductance, $C_L$ say, of this edge is obtained by summing the
conductances of all edges between $V_{L-1}$ and $V_L$ (since they are in
parallel), and thus
\begin{equation}
  C_L = O\bigpar{L^{n-1}}\cdot O\bigpar{e^{-cL^2}} = O(1).
\end{equation}
Consequently, the resistances $C_L^{-1}$ are bounded below, and since $\GGxx$
is just a path with these resistances in series,
\begin{equation}
  \Reff(\GGxx)=\sum_{L=1}^\infty C_L^{-1}=\infty.
\end{equation}
As explained in \refSS{SSelectric}, this implies that $\Reff(\Gabg)=\infty$ and
that the Markov chain $\xi(t)$ is recurrent.
\end{proof}

\begin{lemma}\label{LR1b}
  If\/  $\ga<0$, $\ga+\gb\gl_1<0$ and $-\infty\le\gb\le0$,
 then the CTMC\/ $\xi(t)$ is recurrent.
\end{lemma}
  
\begin{proof}
We use monotonicity. If we replace $\gb$ by 0, then \refL{LR1} applies;
consequently, $\Reff(\Gaog)=\infty$. On the other hand, if $W_0(\xi)$ is
defined by \eqref{W} with $\gb$ replaced by 0, then $W(\xi)\le W_0(\xi)$
(since $\gb\le0$), and thus by \eqref{eqres}, each edge in $\Gabg$ has at most
the same conductivity as in $\Gaog$. Equivalently, each resistance is at
least as large in $\Gabg$ as in $\Gaog$, and thus by
Rayleigh's monotonicity law,
$\Reff(\Gabg)\ge\Reff(\Gaog)=\infty$.
Hence, the Markov chain is recurrent.
\end{proof}

\begin{lemma}\label{LR0}
  If\/  $\ga=0$, $\gb=0$ and $n\le2$,
 then the CTMC\/ $\xi(t)$ is recurrent.
\end{lemma}
  
\begin{proof}
  See \refR{Rsimple} and \refE{Esimple}.
\end{proof}

\begin{lemma}\label{LR2}
  If\/  $\ga=0$, $-\infty\le\gb<0$ and $\kk\le2$,
 then the CTMC\/ $\xi(t)$ is recurrent.
\end{lemma}

\begin{proof}
We assume that $n\ge3$; the case $n\le2$ follows by a simpler version of the
same argument (taking $u=0$ below), or by \refL{LR0} and 
Rayleigh's monotonicity law as in the proof of \refL{LR1b}.

The assumption $\kappa \leq 2$ implies 
that amongst any {\em three} vertices of the graph there are at least two
which are connected by an edge.

Let
$b:=-\beta>0$.
Then, since $\ga=0$, \eqref{W} yields
\begin{equation}
W(\bx)=-\frac{b}{2}\innprod{A\bx,\bx}
=-b\sumij\bx_i\bx_j,\qquad \bx\in\Z_{+}^n.
\end{equation}
Let again $V_L$ be defined  by \eqref{VL},
short-circuit all the vertices within each $V_L$, and denote the resulting
network by $\GGxx$. We can regard each $V_L$ as a vertex of $\GGxx$.

Fix $L\in\Z_+$ and  consider $\bx=(x_1,\dots,x_n)\in V_L$.
Let us order the components of $\bx$ in decreasing order: $x_{(1)}\ge x_{(2)} \ge x_{(3)}\ge \dots \ge x_{(n)}\ge 0$. 
Denote $x_{(3)}=u$; 
then, by construction, $u\in\{0,1,\dots,\lfloor{L/3}\rfloor\}$. 
Among the three vertices corresponding to $x_{(1)},x_{(2)},x_{(3)}$ at least two are connected, so that we can bound
\begin{equation}
  W(\bx)=-b\sumij x_i x_j\leq -bu^2.  
\end{equation}
Hence, by \eqref{eqres}, 
the conductance of each of the resistors coming to $\bx$ from $V_{L-1}$ is bounded above by $e^{-b u^2}$.
Next, the number of such $\bx\in V_L$ with $x_{(3)}=u$
is bounded by  
$n!\, (u+1)^{n-3}L$, 
as there are at most $u+1$ possibilities for each of $x_{(4)},x_{(5)},\dots,x_{(n)}$, at most $L$ possibilities for~$x_{(2)}$ and then
 $x_{(1)}=L-\sum_{i\ge 2}^n x_{(i)}$ is determined, and there are at most
$n!$ different orderings of $x_i$ for each $x_{(1)},\dots,x_{(n)}$.

All these resistors are {\em in parallel}, so 
we sum their conductance to get an effective conductance between $V_{L-1}$ and $V_L$, which is thus bounded above by
\begin{equation}
 n!\,L \sum_{u=1}^L (u+1)^{n-3}  e^{-b u^2}
\le  C(n, b)L,  
\end{equation}
for some $C(n,b)<\infty$.
(Thus, the conductance between $V_{L-1}$ and $V_L$ is of the same order as
in the case $\Z_+^2$ in \refE{Esimple}.)
Hence, the effective resistance $R_L$ between $V_{L-1}$ and $V_L$ is bounded 
below by $cL^{-1}$, and thus
\begin{equation}
  \Reff(\GGxx)=\sumL R_L \ge c\sumL \frac{1}L=\infty.
\end{equation}
Finally, $\Reff(\Gabg)\ge\Reff(\GGxx)=\infty$,
and the chain is therefore recurrent.
\end{proof}

\begin{lemma}\label{LRc}
  If\/  $\ga=0$, $\gb>0$, $e(G)=0$ and $n\le2$,
 then the CTMC\/ $\xi(t)$ is recurrent.
\end{lemma}
  
\begin{proof}
  Since $e(G)=0$, we may replace $\gb$ by 0; the result then follows from
  \refL{LR0}. 
\end{proof}

This completes the classification of transient and recurrent cases.
We proceed to distinguish between positive recurrent and null recurrent
cases;
we do this by analysing the invariant measure $\mu(\xi)=e^{W(\xi)}$,
and in particular its total mass
\begin{equation}
\label{finite}
Z_{\alpha, \beta, G}:=\sum\limits_{\xi\in \Z_{+}^n} e^{W(\xi)}=\sum\limits_{\xi\in \Z_{+}^n} 
e^{\frac{1}{2}\innprod{(\alpha E+\beta A)\xi, \xi}-\ga S(\xi)}\le\infty.
\end{equation}
Note that if $Z=Z_{\alpha, \beta, G}<\infty$, then the invariant measure
$\mu$ can
be  normalised to an invariant distribution 
$Z\qw e^{W(\xi)}$.
Furthermore, recall that an irreducible CTMC is positive recurrent if and
only if it has an 
  invariant distribution and is non-explosive. 

\begin{rem}\label{Rinvariant}
In general, a
 CTMC may have an invariant distribution and 
be explosive (and thus transient),
see e.g.\ \cite[Section 3.5]{Norris}; we will see that this does not
happen in our case. In other words, our CTMC is positive
recurrent exactly when $\Zabg<\infty$.
See also \refS{S:DTMC}.
\end{rem}

\begin{lemma}\label{LI}
  Let $-\infty<\ga<\infty$ and $-\infty\le\gb<\infty$.
Then $\Zabg<\infty$ if and only if
$\ga<0$ and $\ga+\gb\gl_1<0$.
\end{lemma}

\begin{proof}
We consider four different cases.

\pfcasex{$\ga\ge0$}
By \eqref{LT1W}, $e^{W(k\bfe_1)}=e^{\frac{\alpha}{2}k(k-1)}\ge1$, and thus
$\Zabg\ge\sumk e^{W(k\bfe_1)}=\infty$.

\pfcasex{$\ga<0$ and $\ga+\gb\gl_1\ge0$}
Let $y_k$ be as in \refL{LT2}. Then \eqref{WYk} applies and implies in
particular $W(y_k)\ge -C$ for some constant $C$. Hence, 
\begin{equation}
  \Zabg\ge\sumk e^{W(y_k)} \ge \sumk e^{-C}=\infty.
\end{equation}

\pfcasex{$\ga<0$, $\ga+\gb\gl_1<0$ and $\gb\ge0$}\label{cbbc}
The estimate \eqref{CR1} applies for every $\xi\in V_L$, and since the
number of vertices in $V_L$ is $O(L^{n-1})$ for $L\ge1$, we have
\begin{equation}
  \Zabg= 1+\sumL\sum_{\xi\in V_L}e^{W(\xi)}
\le  1+\sumL C_1 L^{n-1} e^{-cL^2}<\infty.
\end{equation}

\pfcasex{$\ga<0$, $\ga+\gb\gl_1<0$ and $-\infty\le\gb\le0$}
We use monotonicity as in the proof of \refL{LR1b}.
Let again $W_0(\xi)$ be given by \eqref{W} with $\gb$ replaced by 0.
Then, since $\gb\le0$, $W(\xi)\le W_0(\xi)$ and thus
$\Zabg\le \Zaog$. Furthermore, $\Zaog<\infty$ by Case
\ref{cbbc}.
Hence, $\Zabg<\infty$.
\end{proof}

\begin{lemma}\label{L+}
  \begin{romenumerate}
  \item \label{L+a}
If\/ $\ga<0$ and $\ga+\gb\gl_1<0$, then the CTMC $\xi(t)$ is positive recurrent.
\item \label{L+b}
If\/ $\ga=0$, $-\infty\le\gb<0$ and $\kk\le2$, then the CTMC $\xi(t)$ is
null recurrent.
\item \label{L+c}
If\/ $\ga=0$, $\gb=0$ and $n\le2$, then the CTMC $\xi(t)$ is null recurrent.
\item \label{L+d}
If\/ $\ga=0$, $\gb>0$, $e(G)=0$ and $n\le2$, then the CTMC $\xi(t)$ is null
recurrent. 
  \end{romenumerate}
\end{lemma}

\begin{proof}
In all four cases, the Markov chain is recurrent, 
by \refLs{LR1}, \ref{LR1b}, \ref{LR0}, \ref{LR2}, \ref{LRc}.
Hence the chain is non-explosive, and the invariant measure is unique up to
a constant factor; furthermore, the chain is positive recurrent if and only
if this measure has finite total mass so that there exists an invariant
distribution. In other words, in these recurrent cases, the chain is
positive recurrent if and only if $\Zabg<\infty$.
By \refL{LI}, this holds in case
  \ref{L+a}, but not in \ref{L+b}--\ref{L+d}.
\end{proof}

\begin{proof}[Proof of \refT{Tmain}]
The theorem follows by collecting \refLs{LT1}--\ref{LT4} and \ref{L+}.
\end{proof}

\section{The corresponding discrete time Markov chain}
\label{S:DTMC}

In this section we consider the discrete time Markov chain (DTMC) 
$\zeta(t)\in \Z_{+}^n$ that corresponds to 
the CTMC $\xi(t)$, i.e.\ the corresponding embedded DTMC. 
Note that  we use $t$ to denote both the continuous and the discrete time,
although the two chains are related by a random change of time.

Recall that the transition  probabilities of DTMC  $\zeta(t)$ are
proportional to corresponding  transition rates of CTMC $\xi(t)$. 
Thus, if the rates of $\xi(t)$ are $q_{\xi,\eta}$, given by \eqref{q},
and $C_{\xi,\eta}=C_{\eta,\xi}$ are the conductances given by \eqref{eqres}
(with $C_{\xi,\eta}=0$ if $\norm{\xi-\eta}\neq1$),
and further $q_\xi:=\sum_{\eta\sim\xi}q_{\xi,\eta}$
and $C_\xi:=\sum_{\eta\sim\xi}C_{\xi,\eta}$,
then
the transition  probabilities 
of $\zeta(t)$ are
\begin{equation}\label{p}
  p_{\xi,\eta}:=\frac{q_{\xi,\eta}}{q_\xi}=\frac{C_{\xi,\eta}}{C_\xi}.
\end{equation}

It is obvious that a CTMC is irreducible if and only if the corresponding
DTMC is, and it is easy to see that the same holds for reversibility.
Similarly, since a CTMC and the corresponding DTMC pass through the same
states (with a random change of time parameter), if one is recurrent [or
transient], then so is the other. However, 
in general, 
since the two chains pass through the states at different speeds, 
one of the chains may be positive recurrent and the other null recurrent.
(Recall that many different CTMC have the same embedded DTMC, and that some
of them may be positive recurrent and others not.)
In our case, there is no such complication.

\begin{theorem}
  \label{TD}
The conclusions in \refT{Tmain} hold also for the DTMC $\zeta(t)$.
\end{theorem}

Before proving the theorem, we note that it follows from \eqref{p} that
the DTMC $\zeta(t)$ is reversible with an invariant measure
\begin{equation}\label{hmu}
  \hmu(\xi):=C_\xi. 
\end{equation}
We denote the total mass of this invariant measure by
\begin{equation}
\label{hfinite}
\hZ_{\alpha, \beta, G}:=\sum_{\xi\in\ZZn} C_\xi
=\sum_{\xi}\sum_{\eta:\,\eta\sim\xi}C_{\xi,\eta}
=2\sum_\xi\sum_{i:\,\xi_i>0} C_{\xi,\xi-\bfe_i}
=2\sum_\xi|\set{i:\xi_i>0}| e^{W(\xi)}.
\end{equation}
Consequently,
\begin{equation}\label{hZ-Z}
  \Zabg-1 \le \hZ_{\alpha, \beta, G} \le 2n\Zabg.
\end{equation}

\begin{lemma}
  \label{LID}
  Let $-\infty<\ga<\infty$ and $-\infty\le\gb<\infty$.
Then $\hZabg<\infty$ if and only if
$\ga<0$ and $\ga+\gb\gl_1<0$.
\end{lemma}
\begin{proof}
Immediate by \eqref{hZ-Z} and \refL{LI}.
\end{proof}

\begin{proof}[Proof of \refT{TD}]
As said above, $\zeta(t)$ is transient precisely when $\xi(t)$ is.

A DTMC is positive recurrent if and only if it has an invariant
  distribution, and then every invariant measure is a multiple of the
  stationary distribution. Hence, $\zeta(t)$ is positive recurrent if and
  only if the invariant measure $\hmu(\xi)$ has finite mass, i.e., if
  $\hZabg<\infty$. \refL{LID} shows that this holds precisely in case
  \ref{Tmain+} of \refT{Tmain}, i.e., when $\xi(t)$ is positive recurrent.
\end{proof}

\begin{rem}
  We can use the DTMC $\zeta(t)$ to 
give an alternative proof of \refL{L+}\ref{L+a} without
  \refLs{LR1}--\ref{LR1b}.
Assume $\ga<0$ and $\ga+\gb\gl_1<0$. Then, by  \refL{LID}, 
$\hZabg<\infty$. Hence, the DTMC $\zeta(t)$ has a stationary distribution
and is thus positive recurrent. (Recall that this implication
holds in general for a DTMC, but not for a CTMC, see \refR{Rinvariant}.)
Hence $\xi(t)$ is recurrent, and thus non-explosive.
Furthermore, \refL{LI} shows that also
$\Zabg<\infty$, and thus also $\xi(t)$ has a stationary distribution.
Since $\xi(t)$ is non-explosive, this implies that $\xi(t)$ is positive
recurrent.
\end{rem}

\section{Explosions}\label{S:ex}

It was shown in~\cite{Volk3} that 
in most of the transient cases in \refT{Tmain}, the CTMC $\xi(t)$ is explosive.
(Recall that a recurrent CTMC is non-explosive.)
We complement this by exhibiting in \refL{Lexpl00} 
one non-trivial transient case where $\xi(t)$ is
non-explosive.

Recall also the standard fact that if, as above, $q_\xi:=\sum_\eta q_{\xi,\eta}$
is the total rate of leaving $\xi$, and $\zeta(t)$ is the DTMC in
\refS{S:DTMC}, then $\xi(t)$ is explosive if and only if
$\sum_{t=1}^\infty q_{\zeta(t)}\qw<\infty$ with positive probability.
In particular, $\xi(t)$ is non-explosive when the rates $q_{\xi}$ are
bounded.

Combining these results, we obtain the following partial classification,
proved later in this section.
Let $\nu_i$ denote the degree of vertex $i\in G$, and note that
\begin{equation}
  \label{nu}
\min_i\nu_i \le \gl_1\le\max_i\nu_i.  
\end{equation}

\begin{theorem}\label{Texpl}
Let $-\infty<\ga<\infty$ and $-\infty\le\gb<\infty$, and 
consider the CTMC\/ $\xi(t)$.
  \begin{romenumerate}
  \item \label{Texpl0}
$\xi(t)$ is non-explosive in the following cases:
    \begin{enumerate}
     \item \label{Texpl00}
$\ga<0$ and $\ga+\gb\gl_1(G)\le0$,
    \item \label{Texpl0b}
$\ga=0$ and $\gb\le 0$,
\item \label{Texpl0c}
$\ga=0$, $\gb>0$ and $e(G)=0$.
    \end{enumerate}

\item \label{Texpl+}
$\xi(t)$ explodes a.s.\ in the following cases:
\begin{enumerate}
\item\label{Texpl+a}  
$\ga>0$,
\item \label{Texpl+b}
$\ga=0$, $\gb>0$ and $e(G)>0$,
\item \label{Texpl+c}
$\ga<0$ and $\ga+\gb\min_i \nu_i>0$.
\end{enumerate}
  \end{romenumerate}
\end{theorem}

\begin{rem}\label{Rexpl}
  \refT{Texpl} gives a complete characterization of 
explosions
when the graph $G$ is regular, i.e., $\nu_i$ is constant, since then 
$\min_i \nu_i=\gl_1$, see \eqref{nu}.

For other graphs $G$,
  \refT{Texpl} leaves one case open,
viz.\
\begin{equation}
\ga<0 
\quad\text{and}\quad
\ga+\gb\min_i \nu_i\le 0 <\ga+\gb\gl_1(G)  
\end{equation}
(and, as a consequence, $\gb>0$).
We conjecture that $\xi(t)$ always 
is explosive in this case, but leave this as an
open problem.
(Our intuition is that in this case, which is transient by \refT{Tmain},
$\xi(t)$ will tend to infinity along a path that stays rather close to the
line $\set{s \bfv_1:s\in\R}$ in $\R^n$, and that the rates $q_\xi$ are
exponentially large close to this line.)
\end{rem}

\begin{lemma}\label{Lexpl00}
  If\/  $\ga<0$ and $\ga+\gb\gl_1(G)=0$,
then the CTMC $\xi(t)$ is transient and non-explosive.
\end{lemma}

We prove first an elementary lemma.

\begin{lemma}\label{Lu}
  Define the functions $\phi,\psi:\R\to\R$ and $\Phi,\Psi:\R^n\to\R$ by,
with $\bfu=(u_1,\dots,u_n)$,
  \begin{align}
    \phi(u)&:=e^u+1,&
\psi(u)&:=u(e^u-1),
\label{psi}\\
\Phi(\bfu)&:=\sumin\phi(u_i),&
\Psi(\bfu)&:=\sumin\psi(u_i).
\label{Psi}
  \end{align}
Then $\Psi(\bfu)/\Phi(\bfu)\to+\infty$ as $\norm{\bfu}\to\infty$.
\end{lemma}

\begin{proof}
Note that $\phi(u)>0$ and $\psi(u)\ge0$ for all $u\in\R$, and that
$\psi(u)/\phi(u)\to+\infty$ as $u\to\pm\infty$.

Fix $B>0$. Then $\psi(u)-B\phi(u)>0$ if $|u|$ is large enough, and thus
there exists a constant $C=C(B)\ge0$ such that $\psi(u)-B\phi(u) \ge -C$
for all $u\in\R$.
Consequently, for any $\bfu\in\R^n$,
\begin{equation}\label{qul}
  B\Phi(\bfu)=\sumin B\phi(u_i)
\le \sumin\bigpar{\psi(u_i)+C}
=\Psi(\bfu)+nC. 
\end{equation}
Furthermore, $\psi(u)\to+\infty$ as $u\to\pm\infty$, and thus
$\Psi(\bfu)\to+\infty$ as $\norm{\bfu}\to\infty$. Consequently, there exists
$M=M(B)$ such that if $\norm{\bfu}>M$, then  $\Psi(\bfu)>nC$, and hence, by
\eqref{qul}, 
$B\Phi(\bfu)< 2 \Psi(\bfu)$, i.e., $\Psi(\bfu)/\Phi(\bfu)>B/2$.
Since $B$ is arbitrary, this completes the proof.
\end{proof}

\begin{proof}[Proof of \refL{Lexpl00}]
The CMTC $\xi(t)$ is transient by \refT{Tmain}\ref{Tmaint-} (\refL{LT2}).

  Let
  \begin{equation}
    \label{Q}
    Q(\xxi):=\frac12\innprod{(\ga E+\gb A)\xxi,\xxi},
\qquad \xxi\in \R^n,
  \end{equation}
be the quadratic part of $W(\xxi)$ in \eqref{W}.
Let, as in \refS{Spf}, $\bfv_1,\dots,\bfv_n$ be an orthonormal basis of
eigenvectors of $A$ with $A\bfv_k=\gl_k\bfv_k$.
The assumptions imply $\gb>0$ and thus, for any $k\le n$,
$\ga+\gb\gl_k\le\ga+\gb\gl_1=0$.
Hence, for any vector $\xxi=\sumkn c_k\bfv_k$,
\begin{equation}\label{Q<=0}
Q(\xxi)=\frac12\sumkn (\ga+\gb\gl_k)c_k^2\le0.
\end{equation}
In other words, $Q(\xxi)$ is a negative semi-definite quadratic form on
$\R^n$.

We denote the gradient of $Q(\xxi)$ by
$U(\xxi)=\bigpar{U_1(\xxi),\dots,U_n(\xxi)}$. Thus, by \eqref{Q},
\begin{equation}
  \label{U}
  U(\xxi):=\grad Q(\xxi) = (\ga E+\gb A)\xxi.
\end{equation}
It follows from \eqref{Q} and \eqref{U} that, for any $\xxi\in\R^n$,
\begin{equation}
  \label{Q+-}
  Q(\xxi\pm\bfe_i)
=Q(\xxi)+Q(\bfe_i)\pm\innprod{(\ga E+\gb A)\xxi,\bfe_i}
=Q(\xxi)+\frac{\ga}2\pm U_i(\xxi).
\end{equation}
Let $\ax:=-\ga/2>0$.
Since $W(\xxi)=Q(\xxi)+\ax S(\xxi)$ by \eqref{W}, and $S(\xxi)$ is the linear
function \eqref{S}, it
follows from \eqref{Q+-} that
\begin{equation}
  \label{W+-}
  W(\xxi\pm\bfe_i)-W(\xxi)
=
  Q(\xxi\pm\bfe_i)-Q(\xxi)
\pm \ax S(\bfe_i)
=\pm U_i(\xxi)-\ax\pm\ax.
\end{equation}
In particular, for $\xi\in\ZZn$, the rate of increase of the $i$-th
component is, by \eqref{q} and \eqref{balance},
\begin{equation}\label{q+}
q_{\xi,\xi+\bfe_i}=  e^{W(\xi+\bfe_i)-W(\xi)}=e^{U_i(\xi)}.
\end{equation}

Fix $\xi\in\ZZn$ and consider the DTMC $\zeta(t)$ started at $\zeta(0)=\xi$,
and the stochastic process $\QZ(t):=Q(\zeta(t))$, $t\in\Z_+$.
Denote the change of $\QZ(t)$ in the first step by
$\DQZ(0):=\QZ(1)-\QZ(0)$.
Then the expected change of $\QZ(t)$ in the first step is,
using \eqref{Q+-}, \eqref{p}, \eqref{q} and \eqref{q+},
\begin{align}
\E\bigpar{\DQZ(0)\mid\zeta(0)=\xi}
&=\sum_\eta p_{\xi,\eta}\bigpar{Q(\eta)-Q(\xi)}
\notag\\&
=\sumin \Bigpar{p_{\xi,\xi+\bfe_i} \bigpar{U_i(\xi)-\ax} 
  +p_{\xi,\xi-\bfe_i} \bigpar{-U_i(\xi)-\ax} } 
\notag\\&
=q_\xi\qw\sumin\Bigpar{e^{U_i(\xi)}U_i(\xi) -\indic_{\xi_i>0} U_i(\xi)}-\ax,
\label{dQ1}
\end{align}
where   $\indic_{\cE}$ denotes the indicator of  an event $\cE$.
Furthermore, 
using \eqref{q+} and the notation \eqref{Psi},
\begin{equation}\label{QPhi}
  q_\xi:=\sumin\bigpar{q_{\xi,\xi+\bfe_i}+q_{\xi,\xi-\bfe_i}}
=\sumin\bigpar{e^{U_i(\xi)}+\indic_{\xi_i>0}}
\le \Phi(U(\xi)).
\end{equation}
(With equality unless some $\xi_i=0$.)
Moreover, if $\xi_i=0$, then \eqref{U} implies 
$U_i(\xi)=\gb\sum_{j\sim i}\xi_j\ge0$. 
Hence,  \eqref{dQ1} yields
\begin{align}
&\E\bigpar{\DQZ(0)\mid\zeta(0)=\xi}
\ge q_\xi\qw \sumin\Bigpar{e^{U_i(\xi)}U_i(\xi) -U_i(\xi)}-\ax
\ge \frac{\Psi(U(\xi))}{\Phi(U(\xi))}-\ax.
\label{dQ2}
\end{align}
\refL{Lu} now implies the existence of a constant $\CCname{\Ca}$ such that if
$\norm{U(\xi)}\ge \Ca$, then
$\E\bigpar{\DQZ(0)\mid\zeta(0)=\xi}\ge0$. 

We have, as in \eqref{Q<=0}, with $\omega_k:=\ga+\gb\gl_k\le0$, the
eigenvalues of $\ga E+\gb A$,
\begin{equation}
  \label{Qgo}
Q(\xxi)=\frac12\sumkn \omega_k \innprod{\xxi,\bfv_k}^2, 
\end{equation}
and it follows that the gradient $U(\xxi)$ can be expressed as
\begin{equation}
  U(\xxi)=\sumkn\omega_k \innprod{\xxi,\bfv_k} \bfv_k,
\end{equation}
and thus
\begin{equation}\label{Ugo}
\norm{ U(\xxi)}^2=\sumkn \omega_k^2 \innprod{\xxi,\bfv_k}^2.
\end{equation}
Comparing \eqref{Qgo} and \eqref{Ugo}, and recalling $\go_k\le0$,
we see that
\begin{equation}\label{UQ}
2 \min_k|w_k|\cdot \abs{Q(\xxi)} \le \norm{U(\xxi)}^2 
\le 2\max_k|w_k|\cdot\abs{Q(\xxi)}.
\end{equation}

Hence, the result above shows the existence of a constant $\CCname\CQ$
such that
\begin{equation}\label{submart}
\text{If}\quad \abs{Q(\xi)}\ge\CQ,
\quad
\text{then}\quad
\E\bigpar{\DQZ(0)\mid  \zeta(0)=\xi}\ge0.
\end{equation}

Fix $m\ge0$ and consider $\zeta(t)$ for $t\ge m$. Define the stopping time
$\tau_m:=\inf\set{t\ge m:\abs{\QZ(t)}\le\CQ}$.
Then \eqref{submart} and the Markov property imply that
the stopped process $-\QZ(t\land\tau_m)$, $t\ge m$,
is a positive supermartingale. 
(Recall that $\QZ(t)\le0$ by \eqref{Q<=0}.)
Hence, this process converges a.s.\ to a finite limit. 
In particular, if we define the events
$\cE_m:=\set{|\QZ(t)|>\CQ\text{ for every $t\ge m$}}$
and
$\cE':=\set{|\QZ(t)|\to\infty\text{ as $t\to\infty$}}$, 
then $\P(\cE_m\cap \cE')=0$. Clearly, $\cE':=\bigcup_{m=1}^\infty
\cE'\cap\cE_m$.
Consequently, $\P(\cE')=0$.

We have shown that a.s.\ $|\QZ(t)|=|Q(\zeta(t))|$ does not converge to $\infty$.
In other words, a.s.\ there exists a (random) constant $M$ such that
$|Q(\zeta(t))|\le M$ infinitely often.
By \eqref{UQ} and \eqref{QPhi}, there exists for each $M<\infty$ a constant
$\CCname\CPhi(M)<\infty$  such that $\abs{Q(\xi)}\le M$ implies
$q_\xi\le\CPhi(M)$. 
Consequently, a.s., $q_{\zeta(t)}\le\CPhi(M)$ infinitely often, and thus
$\sum_{t=0}^\infty q_{\zeta(t)}\qw=\infty$, which implies that $\xi(t)$
does not explode.
\end{proof}

\begin{rem}
Note that  if $\alpha<0$, $\gb>0$  and $\alpha+\lambda_1\beta<0$, 
then the function $Q$ defined in~\eqref{Q} is negative definite, so that $\tilde Q(\bx):=-Q(\bx)\to \infty$ as $\norm{\bx}\to\infty$.
Therefore, it follows from  equation~\eqref{dQ2} that the CTMC $\xi(t)$ is positive recurrent by Foster's 
criterion for positive recurrence (e.g. see~\cite[Theorem 2.6.4]{MPW}). In other words, 
the function $\tilde Q$ can be used as the   Lyapunov function in Foster's criterion for showing 
positive recurrence of the Markov chain in this case.
In fact,  function $\tilde Q$ was used 
in Foster's criterion to show positive recurrence of the Markov chain in the following 
special case
$\alpha<0$ and $\alpha+\beta\max_i\nu_i<0$ in~\cite[Section 4.1.1]{Volk3}.
\end{rem}

\begin{proof}[Proof of \refT{Texpl}]
The non-explosive case  \ref{Texpl00} 
follows from 
\refT{Tmain}\ref{Tmain+} when $\ga+\gb\gl_1(G)<0$ (then the chain is
positive recurrent),
and from \refL{Lu} when $\ga+\gb\gl_1(G)=0$.
The other non-explosive cases \ref{Texpl0b} and \ref{Texpl0c}  
are trivial because in these cases \eqref{q} implies $q_{\xi,\eta}\le1$, and
thus $q_\xi\le 2n$ is bounded.

For explosion, we may assume that $G$ is connected, since we otherwise may
consider the components of $G$ separately, see \refR{Rdisconnected}.
Then,
\cite[Theorem 1(3) and its proof]{Volk3} 
show that if $\ga+\gb\min_i \nu_i>0$ and $\gb\ge0$, then 
$\xi(t)$ explodes a.s.; this includes the cases
\ref{Texpl+b} and \ref{Texpl+c} above, and the case
$\ga>0$, $\gb\ge0$.
Furthermore, \cite[Theorem 2]{Volk3} shows that if
$\ga>0$ and  $\gb\le0$, then $\xi(t)$ a.s.\ explodes;
together with the result just mentioned, this shows explosion when $\ga>0$.
\end{proof}

\begin{rem}
  It is shown in 
\cite{Volk3} that explosion may occur in several different ways, 
depending on both the parameters $\ga,\gb$ and the graph $G$. For example,
if $G$ is a star, then there are (at least) three possibilities, each
occuring with probability 1 when $(\ga,\gb)$ is in some region:
a single component $\xi_i$ explodes (tends to infinity in finite
time); two adjacent components explode simultaneously; or  all
components explode simultaneously. 

Furthermore, the results in~\cite{Volk3} show that in the
 explosive cases in \refT{Texpl+}, 
the Markov chain  asymptotically evolves as a pure birth process,
in the sense that, with probability one, 
there is a random finite time after which none of the components
decreases, 
i.e.\ there are no "death" events after this time.
Consequently, the corresponding discrete time Markov chain can be regarded as a growth process on a graph similar to interacting 
urn   models (e.g., see models in~\cite{Costa}, \cite{Volk1} and~\cite{Volk2}).
One of the main problems in such growth processes is the same as in the urn models. Namely, it is of interest 
 to understand how exactly the process escapes to infinity, i.e.\ whether all components grow indefinitely, or the growth localises 
in a particular subset of the underlying graph. 

We do not discuss this sort of problems here and hope to address it elsewhere. 
\end{rem}

\section{A modified model}\label{Smodified}
In this section, we study the CTMC $\txi(t)$ with the rates $\tq_{\xi,\eta}$
in \eqref{tq}, and the corresponding DTMC $\tzeta(t)$.
This model is interesting since we have ``decoupled'' $\ga$ and $\gb$, with
birth rates depending on $\ga$ and death rates depending on $\gb$.

Since $\tq_{\xi,\xi\pm\bfe_i}$ differ from $q_{\xi,\xi\pm\bfe_i}$ by the
same factor $e^{-\gb\sum_{j:j\sim i}\xi_j}$, which furthermore does not
depend on $\xi_i$, the balance equation \eqref{balanceq} holds for
$\tq_{\xi,\eta}$ too, and thus $\txi(t)$ has the same invariant measure
$\mu(\xi)=e^{W(\xi)}$ as $\xi(t)$.

The electric network $\tGabg$
corresponding to $\txi(t)$ has conductances
\begin{equation}
\label{tC}
\tC_{\xi-\bfe_i, \xi}:=e^{W(\xi-\bfe_i)+\ga (\xi_i-1)}=e^{W(\xi)-\gb (A\xi)_i}.
\end{equation}
\begin{rem}\label{RtC}
  If $\gb>0$, then $\tC_{\xi,\eta}\le C_{\xi,\eta}$,
and  if $\gb<0$, then $\tC_{\xi,\eta}\ge C_{\xi,\eta}$.
(If $\gb=0$, the two models are obviously identical.)
\end{rem}

\begin{theorem}\label{Ttxi}
  The results in \refT{Tmain} hold for $\txi(t)$ too, with a single
  exception:
If $e(G)=1$, $\ga<0$ and $\ga+\gb\gl_1(G)=0$, then $\txi(t)$ is null recurrent
while $\xi(t)$ is transient. 
\end{theorem}
Here $\gl_1(G)$ is as above the largest eigenvalue of $G$.
If $e(G)=1$, then $\gl_1(G)=1$; thus the exceptional case is $e(G)=1$,
$\ga=-\gb<0$.

\begin{proof}
The lemmas in \refS{Spf} all hold for $\txi(t)$ too
by the same proofs with no or minor modifications,
except \refL{LT2} in the case $\ga<0$, $\ga+\gb\gl_1=0$; we omit the
details.
This exceptional case is treated in \refLs{LXt} and~\ref{LX0} below.
\end{proof}
A few cases alternatively follow by \refR{RtC} and the 
Rayleigh monotonicity law.

Before treating the exceptional case, we give a simple combinatorial lemma.

\begin{lemma}
  \label{LG}
Suppose that $G$ is a connected graph with $e(G)\ge2$, and let as above
$\bfv_1=(v_{11},\dots,v_{1n})$ be a positive eigenvector of $A$ with
eigenvalue $\gl_1$.
Then, for each $i$,
\begin{equation}\label{lg}
  v_{1i}<\sum_{j\neq i} v_{1j}.
\end{equation}
\end{lemma}
\begin{proof}
  First, e.g.\ by \eqref{nu}, $\gl_1\ge1$.
Hence, for every $i$,
\begin{equation}\label{lg1}
  v_{1i}\le \gl_1 v_{1i}=(A\bfv_1)_i=\sum_{j\sim i}v_{1j}\le \sum_{j\neq i} v_{1j}.
\end{equation}
If one of the inequalities in \eqref{lg1} is strict, then
\eqref{lg} holds. 
In the remaining case $\gl_1=1$, and every $j\neq i$ is a neighbour of $i$.
Consequently, if $j\neq i$, then
\begin{equation}\label{lg2}
  v_{1j}=\gl_1 v_{1j}=(A\bfv_1)_j 
=\sum_{k\sim j}v_{1k}
\ge v_{1i}.
\end{equation}
By the assumption $e(G)\ge2$, $G$ has at least 3 vertices, and thus
\eqref{lg2}
implies $\sum_{j\neq i} v_{1j} \ge 2 v_{1i}>v_{1i}$, so \eqref{lg} holds in
this case too.
\end{proof}

\begin{lemma}
  \label{LXt}
 If\/ $\ga<0$,  $\ga+\gb\gl_1\ge0$ and $e(G)\ge2$, 
 then the CTMC\/ $\txi(t)$ is transient.
\end{lemma}

\begin{proof}
If $G$ is connected, then $\bfv_1$ satisfies \eqref{lg} by \refL{LG}.  

On the other hand, if $G$ is disconnected and has a component with at least
two edges, it suffices to consider that component.

In the remaining case, $G$ consists only of isolated edges and
vertices. There are at least two edges, which we w.l.o.g.\ may assume are 12
and 34. Then $\gl_1=1$ and 
$\bfv_1:=\frac12(\bfe_1+\bfe_2+\bfe_3+\bfe_4)=\frac12(1,1,1,1,0,\dots)$ is
an eigenvector satisfying  \eqref{lg}.

Hence we may assume that $\bfv_1$ satisfies \eqref{lg}. Hence there exists
$\gd>0$ such that for every $i=1,\dots,n$,
\begin{equation}\label{lgg}
S(\bfv_1)=  v_{1i}+\sum_{j\neq i} v_{1j} \ge 2v_{1i}+\gd.
\end{equation}

We follow the proof of \refL{LT2}, and note that there is equality in
\eqref{WYk0} and \eqref{WYk}. Hence, for any $i$, again writing
$\ax:=-\ga/2>0$, and using \eqref{yx},
\begin{align}
  W(y_k)+\ga y_{k,i}
&=\ax\bigpar{S(y_k)-2y_{k,i}}+O(1)
=\ax t_k \bigpar{S(\bfv_1)-2v_{1,i}} + O(1)
\notag\\&
\ge \ax\gd t_k + O(1) 
\ge c' k + O(1)
\end{align}
for some $c>0$.
Thus, the resistance of the edge connecting $y_{k}$ and
$y_{k+1}$ is, for some $i$, 
\begin{equation}\label{lgz}
  R_{k+1}=\tC_{y_k,y_k+\bfe_i}\qw = e^{-W(y_k)-\ga y_{k,i}} \le e^{-c' k+O(1)}.
\end{equation}
Hence, $\sumk R_k <\infty$, and the network is transient by the same
argument as before.
\end{proof}

\begin{lemma}
  \label{LX0}
 If\/ $\ga<0$,  $\ga+\gb\gl_1\ge0$ and $e(G)=1$, 
 then the CTMC\/ $\txi(t)$ is null recurrent.
\end{lemma}

\begin{proof}
  Suppose first that $n=2$ so $G=\sfK_2$ consists of a single edge.
Then \eqref{W} gives, with $\ax:=-\ga/2>0$ as above,
\begin{equation}
  W(\xi_1,\xi_2)=-\ax(\xi_1-\xi_2)^2+\ax(\xi_1+\xi_2),
\end{equation}
and then \eqref{tC} yields
\begin{equation}
  \tC_{\xi,\xi+\bfe_1}= e^{W(\xi)+\ga\xi_1} 
=e^{-\ax(\xi_1-\xi_2)^2+\ax(\xi_2-\xi_1)}\le1,
\end{equation}
and similarly, $  \tC_{\xi,\xi+\bfe_1}\le1$.
Hence all conductances are bounded by 1, and thus all resistances are
bounded below by 1. We may compare the network $\tGabg$ to the network
$\Z_+^2$ with unit resitances, and obtain by 
Rayleigh's monotonicity law $\Reff(\tGabg)\ge\Reff(\Z_+^2)=\infty$,
recalling that simple random walk on $\Z_+^2$ is recurrent by \refR{Rsimple}
and \refE{Esimple}. Hence $\txi(t)$ is recurrent.

The invariant measure $e^{W(\xi)}$ is the same as for $\xi(t)$ and
has total mass $\Zabg=\infty$ by \refL{LI}; hence $\txi(t)$ is not positive
recurrent. 

This completes the proof when $G$ is connected. 
If $G$ is disconnected, then $G$ consist of
one edge and one or several isolated vertices. By \refR{Rdisconnected}, 
$\xi(t)$ then consists of $n-1$ independent parts: one part is the CTMC
in $\Z_+^2$ defined by the graph $\sfK_2$, which is null recurrent by the
first part of the proof; the other parts are independent copies of the
CMTC in $\Z_+$ defined by a single vertex, and these are positive recurrent
since $\ga<0$. It is now easy to see that the combined $\xi(t)$ 
is null recurrent.
\end{proof}

The corresponding DTMC $\tzeta(t)$ has invariant measure
\begin{equation}\label{tmu}
  \tC_\xi:=\sum_\eta \tC_{\xi,\eta}.
\end{equation}
Note that this (in general) 
differs from the invariant measure $C_\xi$ for $\zeta(t)$,
see \eqref{hmu}.
Denote the total mass of this invariant measure by
\begin{equation}
  \tZabg:=\sum_\xi \tC_\xi=\sum_{\xi,\eta} \tC_{\xi,\eta}.
\end{equation}
There is no obvious analogue of the relation \eqref{hZ-Z}, but we can
nevertheless prove the following analogue of \refL{LID}
\begin{lemma}
  \label{LIDt}
  Let $-\infty<\ga<\infty$ and $-\infty\le\gb<\infty$.
Then $\tZabg<\infty$ if and only if
$\ga<0$ and $\ga+\gb\gl_1<0$.
\end{lemma}
\begin{proof}
By the proof of  \refL{LI} with minor modifications. In particular, in the
case $\ga<0$ and $\ga+\gb\gl_1=0$, we argue also as in
\eqref{lgg}--\eqref{lgz}
in the proof of
\refL{LXt} (but now allowing $\gd=0$). 
We omit the details.
\end{proof}

\begin{theorem}\label{Ttzeta}
  \refT{Ttxi} holds for the DTMC $\tzeta(t)$ too.
\end{theorem}

\begin{proof}
  By \refT{Ttxi} for recurrence vs transience, and by \refT{Ttzeta} for
  positive recurrence vs null recurrence.
\end{proof}
We are not going to analyse the modified model any further.

\section{Alternative proofs using Lyapunov functions}
\label{Lyap}

In this  section we give alternative proofs of 
some parts of Theorem~\ref{Tmain}.
These proofs do not use reversibility, and
have therefore potential extensions also to cases where electric networks
are not applicable.
They are based on the following 
recurrence criterion for countable Markov chains using Lyapunov functions,
see e.g.\ \cite[Theorem 2.2.1]{FMM}.

\begin{criterion}\label{rc}
 A CTMC with values in  $\Z_{+}^n$ is recurrent if and only if 
 there exists a positive function $f$ (the Lyapunov function) 
on $\Z_{+}^n$ such that 
$f(\xi)\to \infty$ as $\xi\to \infty$ and\/ 
${\sf L}f(\xi)\leq 0$ for all $\xi\notin D$, where ${\sf L}$ is the Markov chain generator, and 
$D$ is a finite set. 
\end{criterion}

Note that the Lyapunov function $f(\xi)$ is far from unique. The idea of the
method is to find some explicit function $f$ for which the conditions can be
verified. 
There is also a related criterion for transience \cite[Theorem 2.2.2]{FMM},
but we will not use it here.

We give only some examples. 
(See also \cite{Volk3} for further examples.)
It might be possible to give a complete proof of
\refT{Tmain} using these methods, but this seems rather challenging.
Note that (since our Markov chains have bounded steps), the Lyapunov
function $f$ can be changed arbitrarily on a finite set; hence it suffices
to define $f(\xi)$ (and verify its properties) for $\norm\xi$ large.
We do so, usually without comment, in the examples below.

\begin{example}
\exx{Proof of the hard-core case of \refT{Tmain}\ref{Tmain0a} 
by the \refrc{rc}.}
Assume that  $\alpha=0$, $\beta=-\infty$ and    $k_{\max}(G)$ $\le 2$. 
As said in \refSS{SShardcore}, we may assume that the Markov chain lives on
$\GO$ defined in \eqref{GO}; since $\kk\le2$, this implies that
  no more than two  components of the process can be non-zero.
Therefore, the  Markov chain evolves as a simple random walk on a certain finite union of quadrants of $\Z_+^2$ and half-lines $\Z_+$ glued along the axes. 
Each of these random  walks is null-recurrent, and, hence, the whole process
should be null-recurrent as well.
 We provide a rigorous justification to this heuristic argument by 
using the \refrc{rc}.

The generator ${\sf L}$ of the Markov chain in the  case $\alpha=0$, $\beta=-\infty$
is 
$${\sf L}f(\xi)=\sum_{i=1}^n \left(f\left(\xi+e_i\right)-f(\xi)\right)\indic_{\{\xi: \xi_j=0, \, j\sim i\}}
+\left(f\left(\xi-e_{i}\right)-f(\xi)\right)\indic_{\{\xi_i>0\}}.
$$
We define a Lyapunov function on $\Z_{+}^n$ by
\begin{equation}
\label{f}
f(\xi):=
\log\left(\norm{\xi-{\bf e}}^2 -n+\frac 32\right),  
\qquad\norm{\xi}\geq C_1,
\end{equation}
where  ${\bf e}=(1,\dots,1)\in\Z_{+}^n$ is the vector whose all coordinates
are equal to $1$, and
 $C_1>0$ is sufficiently large  so that  the expression inside the log is
 greater than 1. 
Note that 
the function is defined for any state in $\ZZn$, but we consider it only on
the subset $\GO$. 
Let $\xi\in\GO$ with $\norm{\xi}>C_1+1$.
First, assume that $\xi$ has two non-zero components, say $x>0$ 
and $y>0$, so that 
$$f(\xi)=\log\left((x-1)^2+(y-1)^2-\frac12\right),$$
and both $x$ and $y$ can increases as well as decrease. A direct  computation gives that 
\begin{align*}
{\sf L}f(\xi)&=\log\left(x^2+(y-1)^2-\frac12\right)
+\log\left((x-1)^2+y^2-\frac12\right)\\
&\qquad
+\log\left((x-2)^2+(y-1)^2-\frac12\right)
+\log\left((x-1)^2+(y-2)^2-\frac12\right)\\
&\qquad
-4\log\left((x-1)^2+(y-1)^2-\frac12\right)
\\
&=\log\left(1-\frac{64(x-y)^2(x+y-2)^2}{(2x^2+2y^2-4x-4y+3)^4}
\right) \le 0.
\end{align*}
Next, assume that $\xi$ has only one non-zero component, say  $x=a+1>0$. Then
$f(\xi)=\log\left(a^2+1/2\right),$
and  this component  can both increase and decrease.  Note that {\it some} of the other components may also increase by~$1$, 
and assume there are $m\ge 0$ such components. 
A direct computation gives that 
\begin{align*}
{\sf L}f(\xi)&=\left[\log\left((a+1)^2+\frac12\right)
+\log\left((a-1)^2+\frac12\right)
-2\log\left(a^2+\frac12\right)\right]\\
&+m\left[\log\left(a^2-\frac12\right)-\log\left(a^2+\frac12\right)\right]\\
&=\log\left(\frac{4a^4-4a^2+9}{4a^4+4a^2+1}\right)+m\left[\log\left(\frac{2a^2-1}{2a^2+1}\right)\right]\leq 0.
\end{align*}
Hence, $\gen f(\xi)\le0$ whenever $\xi\in\GO$ with $\norm{\xi}>C_1+1$.
It follows now from the recurrence criterion \ref{rc}
that CTMC $\xi(t)$  is recurrent. 
\end{example}

Now consider the case  $\alpha=0$ and  $-\infty<\beta<0$.
The generator of  the Markov chain with parameter $\alpha=0$ is 
\begin{equation}
\label{generator}
\gen f(\xi)=\sum_{i=1}^n \left(f\left(\xi+e_i\right)-f(\xi)\right)e^{\beta(A\xi)_i}
+\left(f\left(\xi-e_{i}\right)-f(\xi)\right)\indic_{\{\xi_i>0\}}.
\end{equation}

We consider for simplicity only some small graphs $G$,
using modifications of the
Lyapunov function (\ref{f}) used in the hard-core case.

Recurrence in  the case  
$\alpha=0$, $b:=-\beta>0$ and
$G=\sfK_2$, the graph with just  $2$  vertices and
a single edge, 
was shown in~\cite{Volk3} by applying the \refrc{rc} with 
the Lyapunov function $f(\xi)=\log(\xi_1+\xi_2+1)$.
Alternatively, one could use e.g.\
$f(\xi)=\log(\xi_1+\xi_2)$
or
$\log(\xi_1^2+\xi_2^2)$.
We extend this to the case $G=\sfK_n$, the complete
graph with $n$ vertices, for any $n\ge2$.

 \begin{example}\label{EKn}
\exx{Recurrence in the case $\alpha=0$, $\beta=-b<0$ and $G=\sfK_n$.}
We use the function $f(\xi):=\log\Vert \xi\Vert$.
(Similar arguments work for variations such as $\log\bigpar{\norm\xi^2\pm1}$
and $\log(\xi_1+\dots+\xi_n)$.)

Regard $f(\xi)$ as a function on $\R^n\setminus\set0$ and write $r=\norm{\xi}$.
The partial derivatives of $f$ are
\begin{equation}
  \frac{\partial f(\xi)}{\partial \xi_i}=\frac{\xi_i}{r^2}
\end{equation}
and all second derivatives are $O\bigpar{r\qww}$.
Hence, a Taylor expansion of each of the differences in \eqref{generator}
  yields the formula, for $\xi\in\ZZn$,
\begin{equation}
\label{qa}
\gen f(\xi)
=\sum_{i=1}^n\frac{\xi_i}{r^2}\bigpar{e^{-b(A\xi)_i}-\indic_{\{\xi_i>0\}}}
+O\bigpar{r\qww}
=\sum_{i=1}^n\frac{\xi_i}{r^2}\bigpar{e^{-b(A\xi)_i}-1}
+O\bigpar{r\qww}.
\end{equation}

Suppose first that at least 2 components $\xi_i$ are positive. Then
$A\xi_i=\sum_{j\neq i}\xi_j\ge1$ for every $i$, and thus \eqref{qa} implies,
since $r\le\sum_i\xi_i$,
\begin{equation}
\label{qb}
\gen f(\xi)
\le -\bigpar{1-e^{-b}}\sum_{i=1}^n\frac{\xi_i}{r^2}+O\bigpar{r\qww}
\le -\frac{1-e^{-b}}r+O\bigpar{r\qww},
\end{equation}
which is negative for large $r$, as required.

It remains to consider the case when a single component $\xi_i$ is positive,
say $\xi=(x,0,\dots,0)$ with $x>0$.
Then the estimate \eqref{qa} is not good enough.
Instead we find from \eqref{generator}
\begin{align}
\gen f(\xi)
&=
\log(x+1)+\log(x-1)-2\log x+(n-1)e^{-bx}\bigpar{\log\sqrt{x^2+1}-\log x}
\nonumber\\&
=
\log\frac{x^2-1}{x^2} +\frac{n-1}2e^{-bx}\log\frac{x^2+1}{x^2}
\le -\frac{1}{x^2}+\frac{n-1}2e^{-bx},
\label{q1}
\end{align}
which is negative when $x$ is large.

Hence, in both cases, $\gen f(\xi)\le0$ when $\norm\xi$ is large, and
recurrence follows by the \refrc{rc}.
\end{example}

The argument in \refE{EKn} used the fact that $G$ is a complete graph
  so that $(A\xi)_i\ge1$ unless only $\xi_i$ is non-zero.
Similar arguments work for some other graphs.

\begin{example}
Let again $\ga=0$, $\gb<0$ and let
$G=\sfK_{1,2}$, a star with 2 non-central vertices, which is the same as
a path of 3 vertices. Number the vertices with the central vertex as 3, and
writre $\xi=(x,y,z)$.
Taylor expansions similar to the one in \eqref{qa}, but going
further, show that $f(\xi):=\log\norm\xi$ is not a Lyapunov function.
(The problematic case is $\xi=(x,x,0)$, with 
$\gen f(\xi)=\frac12r^{-4}+O(r^{-6})$.)
However, similar calculation also show that
$f(\xi):=\log(\norm\xi^2-1)$ is  a Lyapunov function, showing recurrence by 
the \refrc{rc}. We omit the details.
\end{example}

\appendix
\section{An alternative argument in the hard
core case}

We present here yet another argument, which seems to be able to give an
alternative proof of \refT{Tmain}\ref{Tmaintkk}. 
However, the argument is not completely rigorous,
so it should in its present form be regarded only as a heuristic argument.
It is included here in order to suggest future developments.

For simplicity, we consider  a particular case, namely, when the graph
$G$ is a cycle $\sfC_6$ with $6$ vertices. 
 In this case the state space of the Markov chain is 
 $\GO=\bigcup_{1}^3\Gamma^2_j \cup \bigcup_{1}^2\Gamma^3_j$, 
where (with addition of indices modulo 6)
\begin{align*}
\Gamma_j^2&=\bigset{(x_1,\dots,x_6): x_i=0 \text{ unless } i\in\{j,j+3\}},
\\
\Gamma_j^3&=\bigset{(x_1,\dots,x_6): x_i=0 \text{ unless } i\in\{j,j+2,j+4}\}.  
\end{align*}
The Markov chain is simple random walk on this state space.

If the random walk is  in $\Gamma^3_j$, then it will sooner or later reach a
point where two of the three allowed coordinates are $0$, and thus only one
non-zero, say $x_j$. The random walk then may either go back into
$\Gamma^3_j$, or it may move into $\Gamma^2_j$. In the latter case, it might
either return again to the intersection line $L_j=\Gamma^2_j\cap\Gamma^3_j$,
or it might cross $\Gamma^2_j$ and reach $L_{j+3}$, in which case it may go
on to other parts of $\Gamma$. 
We want to show that  with positive probability, the latter case will not
happen.  
Consequently, the random walk will a.s.\ eventually be confined to 
$\widehat\gG^3_j:=\Gamma^3_j\cup\Gamma^2_j\cup\Gamma^2_{j+2}\cup\Gamma^2_{j+4}$
for $j=1$ or $2$;  
in particular, the random walk is transient.

If the random walk is in $\widehat\gG^3_j$, it may escape through $\gG^2_j$,
$\gG^2_{j+2}$ or $\gG^2_{j+4}$.
Allowing three routes of escape does not seem to be significantly different
from just one, so we consider for simplicity instead random walk on
$\Gamma'_j:=\Gamma^3_j\cup\Gamma^2_j$. (This is one of the non-rigorous
steps.)
The rest of the argument is thus devoted to showing
the following, which implies that there is a positive
probability of not escaping.
\begin{claim}
\label{CL2}
A random walk on $\Gamma_j'$ a.s.\ hits the line $L_{j+3}\subset\Gamma^2_j$ only a finite number of times. 
\end{claim}
We drop the index $j$. Note that $\Gamma'$ is a product $\Gamma'=U\times (V\cup W)$, where $U\cong\N$, $V\cong\N$ and $W\cong\N^2$, and  $V$ and $W$ intersect in the single point  $0$ (the 0 in both $V$ and $W$).

We consider continuous time, with jumps with rate $1$ along any edge. Then the random walk consists of two independent components, a continuous-time random walk in $U$ and  a continuous-time random walk in $V\cup W$.

Consider the latter (i.e.\ a continuous-time random walk in $V\cup W$).  It
returns infinitely often to 0, making excursions into either $V$ or $W$; the
excursions are independent. (Recall that random walk in $V$ or in $W$ is
recurrent, thus every excursion is finite and eventually returns to 0.)

Consider first excursions into $V\cong\N$. Let $T_V$ be the time until the first return, $f_V(s):=\E e^{-s T_V}$ its Laplace transform, and 
\begin{equation}
    g_V(s):=\int_0^\infty e^{-sx}\rmd\mu_V(x)=\frac{f_V(s)}{1-f_V(s)},
\end{equation}
the Laplace transform of the corresponding renewal measure $\mu_V$. By
symmetry, we can consider a random walk on $\Z$ instead of $\Z_{+}$, and
then the intensity $\rmd\mu_V(x)/\rmd x$ of a return at $x$   is $\approx
x^{-1/2}$,
where $\approx$ means 'of the same order as'.
Hence, as $s\to 0$,
\begin{equation}
 g_V(s)\approx \int_0^\infty e^{-sx}x^{-1/2}\rmd x \approx s^{-1/2},  
\end{equation}
and thus $f_V(s)\approx 1-s^{1/2}.$

For excursions into $W\cong\N^2$, we similarly have $\rmd\mu_W(x)/\rmd x\approx x^{-1}$  
and hence, for small $s$,
$g_W(s)\approx \int_1^\infty e^{-sx}x^{-1}\rmd x \approx |\log s|$ and thus 
\begin{equation}
    f_W(s)=\frac{g_W(s)}{1+g_W(s)}\approx
  \frac{|\log(s)|}{1+|\log(s)|}
  \approx 1-1/|\log s|.
\end{equation}
For the combined excursions, we have 
$
f_{V\cup W}(s) = \frac{1}3f_V(s) +\frac{2}3f_W(s)  
$
and thus
$$
1-f_{V\cup W}(s) = \frac{1}3\left(1-f_V(s)\right)
 +\frac{2}3\left(1-f_W(s)\right)
\approx \frac{2}3\left(1-f_W(s)\right)\approx 1/|\log s|,
$$
since $1-f_W(s)\gg 1-f_V(s)$. Hence, the Laplace transform of the renewal
measure $\mu_{V\cup W}$ describing the intensity of returns to 0 in $V\cup W$  is 
\begin{equation}\label{gVW}
  g_{V\cup W}(s)=\frac{f_{V\cup W}(s)}{1-f_{V\cup W}(s)}
\approx |\log s|.
\end{equation}

Now consider also $U$. We are interested in excursions into $\Gamma^2=U\times V$ that cross to the other boundary. This happens if the random walk in $U$ hits $0$ during an excursion of the $V\cup W$ walk into $V$. 

Consider an excursion into $\Gamma^2=U\times V\cong\N^2$ starting at time
$t$; it 
begins with a step from $(X_U(t),0)$ to $(X_U(t),1)$, where $X_U$ is the
random walk on $U$. By \refL{LP1} below,
the probability that a random walk in $\Z_{+}^2$ starting at $(x,1)$ hits
the diagonal before it hits $\N\times\{0\}$ is $\approx 1/(x+1)$. Hence, the
probability $q(t)$ that an excursion starting at time $t$ hits the diagonal
in
$\Gamma^2$ is $q(t)  \approx \E \frac{1}{X_t+1}$.

Now, $X_t$ is (almost) the same as the modulus of a simple random walk on
$\Z$, so at time $t$, by the Central Limit Theorem, 
the probability function $\P(X_t=k)=O(t^{-1/2})$, uniformly in $k$. Hence,
for large $t$,
\begin{equation}
  q(t) \approx  \E \frac{1}{X_t+1}
=\sum_{k=0}^{\sqrt t} \frac{\P(X_t=k)}{k+1} + O(t^{-1/2})
\approx \sum_{k=0}^{\sqrt t} \frac{t^{-1/2}}{k+1} + O(t^{-1/2})
\approx t^{-1/2} \log t = O(t^{-0.49}).	
\end{equation}
The expected number of times an excursion into $U\times V$ hits the
diagonal before it hits $\N\times\{0\}$ is thus
\begin{align*}
\int_0^{\infty} q(t)\rmd\mu_{V\cup W}(t)
&\le C \int_0^{\infty} t^{-0.49}\rmd\mu_{V\cup W}(t)
\approx  C \int_0^{\infty} \int_0^1 s^{-0.51}e^{-st}\rmd s \rmd\mu_{V\cup W}(t) 
 \\& = C \int_0^1  s^{-0.51}g_{V\cup W}(s) \rmd s 
 \approx  \int_0^1  s^{-0.51}|\log s| \rmd s <\infty,
\end{align*}
where we used the fact that
\begin{align*}
\int_0^1 s^{-0.51} e^{-st} \rmd s=
\int_0^{\infty} s^{-0.51} e^{-st} \rmd s
-\int_1^{\infty} s^{-0.51} e^{-st} \rmd s
=\Gamma(0.49)\cdot t^{-0.49}+o(t^{-1})
\end{align*}
and hence, for $t\ge1$, say, 
$t^{-0.49}\approx \int_0^1 s^{-0.51} e^{-st} \rmd s$.
Consequently, a.s.\ only a finite number of excursions into $\gG^2$ will hit
the diagonal. Any excursion hitting the line $L_{j+3}=\set{0}\times V$ has to hit
the diagonal first, and thus there is a.s.\ only a finite number of such
excursions, each hitting the line a finite number of times.

This completes our (partly heuristic) argument for Claim \ref{CL2}, and thus for
transience of the Markov chain.

\begin{lemma}
\label{LP1}
Let $(X_n, Y_n)$ be a discrete time symmetric simple random walk on
$\Z_+^2$. Then, for every $x>0$,
\begin{equation}
\frac{1}{x}\leq \P\bigpar{(X_n, Y_n)  
\mbox{ hits the diagonal before it hits } y=0 \mid
(X_0, Y_0)=(x, 1)}
\leq \frac{2}{1+x}.
\end{equation}
\end{lemma}

\begin{proof}
Define $\tau:=\min\set{n: Y_n=0 \text{ or } X_n=Y_n}$, 
$Z_n:=\frac{Y_n}{X_n+Y_n}$ and $\tilde Z_n:=Z_{n\wedge \tau}$. A direct
computation gives that $\E(Z_{n+1}- Z_n\mid(X_n, Y_n))<0$  if
$0<Y_n<X_n$. 
Hence, the stopped process $\tilde Z_n$ is a bounded supermartingale.
Furthermore, $\tau<\infty$ a.s., and it follows from
the Optional Stopping Theorem that
$\E\bigpar{Z_\tau\mid(X_0, Y_0)=(x, 1)}\leq \tilde Z_0=\frac{1}{x+1}$.
On the other hand, 
$Z_\tau$ takes only the values $0$ and $1/2$, with the latter value if the
diagonal is hit first. Thus
$\E(Z_\tau)=\frac{1}{2}\P\left(Z_\tau=\frac{1}{2}\right)$. 
Therefore, given $(X_0, Y_0)=(x, 1)$,  the probability that $(X_n, Y_n)$ hits diagonal before the line $y=0$ is no larger than $\frac{2}{1+x}$.

Consider now  the process $\tilde W_n:=W_{\tau\wedge n}$, where  $W_n:=\frac{Y_n}{X_n}$. A direct computation  gives that 
$\E\left(\tilde W_{n+1}-\tilde W_n\mid(X_n, Y_n)\right)\ge0$ is
non-negative. Thus $\tilde W_n$ is a bounded submartingale, and
by the optional stopping theorem $\E(W_\tau)\geq \tilde W_0=\frac1x$.
Furthermore, $\E(W_\tau)=\P(W_{\tau}=1)$, the probability of hitting the
diagonal.
Therefore, given $(X_0, Y_0)=(x, 1)$, the probability that $(X_n, Y_n)$ hits 
the diagonal before it hits the line $y=0$ is at least $\frac1x$. 
\end{proof}

\subsection*{Acknowledgement}
S.J.~research is partially supported by 
the Knut and Alice Wallenberg Foundation.
\\
S.V.~research is partially supported by Swedish Research Council grant
VR2014--5157.
\\
 We thank James Norris for helpful comments.

\end{document}